\documentclass[a4paper,12pt,leqno]{amsart}
\usepackage{amsmath,amssymb,amscd, amsthm}
\usepackage{epic,eepic,epsfig}
\usepackage{comment} 
\usepackage{mathrsfs}
\usepackage[all,cmtip]{xy}
\usepackage{graphicx}
\usepackage{here}
\makeatletter
\def\@settitle{\begin{center}%
  \baselineskip14\p@\relax
  \normalfont\LARGE\bfseries
  \@title
  \ifx\@subtitle\@empty\else
     \\[1ex] 
     \normalsize\mdseries\@subtitle
  \fi
 \ifx\@didication\@empty\else
     \\[2ex] 
     \large\mdseries\it\@dedication
  \fi
  \end{center}%
}
\def\subtitle#1{\gdef\@subtitle{#1}}
\def\@subtitle{}
\def\dedication#1{\gdef\@dedication{#1}}
\def\@dedication{}
\renewcommand{\section}{\@startsection
{section}{1}{0mm}{5mm}{2mm}{\raggedright\bfseries}}

  \@addtoreset{equation}{section}
\makeatother

\newtheorem{theorem}{Theorem}[section] 
\newtheorem{Theorem}[theorem]{Theorem}
\newtheorem{Lemma}[theorem]{Lemma}
\newtheorem{Corollary}[theorem]{Corollary}
\newtheorem{Proposition}[theorem]{Proposition}

\theoremstyle{definition}
\newtheorem{Definition}[theorem]{Definition}

\newtheorem{Remark}[theorem]{Remark}

\newtheorem*{Example*}{Example}
\newtheorem*{Claim*}{Claim}
\newtheorem*{Question*}{\it Question}

\begin{document}

\def\check{{\clubsuit}}
\def\Z{{\mathbb Z}}
\def\G{{\mathbb G}}
\def\C{{\mathbb C}}
\def\Q{{\mathbb Q}}
\def\R{{\mathbb R}}
\def\N{{\mathbb N}}
\def\Gal{\mathrm{Gal}}
\def\vru{\,\vrule\,}
\def\wf{{f}}
\def\ff{\mathfrak{f}}
\def\et{\text{\'et}}
\def\ab{\mathrm{ab}}
\def\check{{\clubsuit}}
\def\boldzeta{{\boldsymbol \zeta}}
\def\nyoroto{{\rightsquigarrow}}
\newcommand{\pathto}[3]{#1\overset{#2}{\dashto} #3}
\newcommand{\pathtoD}[3]{#1\overset{#2}{-\dashto} #3}
\def\dashto{{\,\!\dasharrow\!\,}}
\def\ovec#1{\overrightarrow{#1}}
\def\isom{\,{\overset \sim \to  }\,}
\def\Isom{\mathrm{Isom}}
\def\proP{{\text{pro-}p}}
\def\padic{{p\mathchar`-\mathrm{adic}}}
\def\la{\langle}
\def\ra{\rangle}
\def\scM{\mathscr{M}}
\def\scLi{{\mathscr{L}i}}
\newcommand{\Li}{\mathrm{Li}}
\newcommand{\cC}{\mathcal{C}}
\def\tilbchi{\tilde{\boldsymbol \chi}}
\def\tilchi{{\tilde{\chi}}}
\def\bkappa{{\boldsymbol \kappa}}
\def\lala{\la\!\la}
\def\rara{\ra\!\ra}
\def\ttx{{\mathtt{x}}}
\def\tty{{\mathtt{y}}}
\def\ttz{{\mathtt{z}}}
\def\kk{{\varkappa}}     

\title{On adelic Hurwitz zeta measures}

\author{Hiroaki Nakamura and Zdzis{\l}aw Wojtkowiak}


\address{Hiroaki Nakamura: 
Department of Mathematics, 
Graduate School of Science, 
Osaka University, 
Toyonaka, Osaka 560-0043, Japan}
\email{nakamura@math.sci.osaka-u.ac.jp}

\address{Zdzis{\l}aw Wojtkowiak: 
Laboratoire de Mathématiques J.A. Dieudonné,
UMR n${}^\circ$7351 CNRS UNS,
Université de Nice - Sophia Antipolis,
06108 Nice Cedex 02,
France}
\email{wojtkow@math.unice.fr}

\maketitle

\markboth{H.Nakamura and Z.Wojtkowiak}
{On adelic Hurwitz zeta measures}
\begin{abstract}
In this paper we construct a $\hat\Z$-valued measure on
$\hat\Z$ which interpolates $p$-adic  Hurwitz zeta functions 
for all $p$.
\end{abstract}

\setcounter{tocdepth}{1}
\tableofcontents



\section{Introduction}

Let $m\ge 1$, $0<a<m$ be integers such that $a$ is prime to $m$,
and let $p$ be a rational prime.
Set $q:=4$, $q:=p$ according to whether $p=2$ or $p>2$ respectively,
and $e:=|(\Z/q\Z)^\times|$.
When $p\nmid m$, let $\la a p^{-1}\ra$ denote
the least positive integer such that
$\la  a p^{-1}\ra p\equiv a$ mod $m$.
Define the Bernoulli polynomials $B_k(T)$ $(k\in\N)$
by
$\sum_{k=0}^\infty B_k(T)\frac{w^k}{k!}=\frac{we^{Tw}}{e^w-1}$
and set the Bernoulli numbers $B_k:=B_k(0)$.

In \cite{Sh}, Shiratani constructed $p$-adic Hurwitz zeta functions
$\zeta_p^{Sh}(s;a,m)$ ($s\in\Z_p$, $s\ne 1$) characterized by the interpolation property:
\begin{equation} \label{Shiratani}
\zeta_p^{Sh}(1-k;a,m)=
\begin{cases}
-\frac{m^{k-1}}{k}B_k(\frac{a}{m}), &(p\mid m); \\
-\frac{m^{k-1}}{k}B_k(\frac{a}{m})+p^{k-1}\frac{m^{k-1}}{k}B_k(\frac{\la a p^{-1}\ra}{m}), &(p\nmid m)
\end{cases}
\end{equation}
for all integers $k>1$ with $k\equiv 0$ mod $e$.
In \cite{W3}, assuming $p\nmid m$,
the second author
introduced a $p$-adic Hurwitz $L$-function
$L_p^\beta(s;a,m)$ for $\beta\in(\Z/e\Z)$ 
which satisfies
\begin{equation} \label{WojtowiakIstanbul}
L_p^\beta(1-k;a,m)=
\frac{1}{k}B_k\left(\frac{a}{m}\right)-
\frac{p^{k-1}}{k}B_k\left(\frac{\la a p^{-1}\ra}{m}\right)
\end{equation}
for all integers $k>1$ with $k\equiv \beta$ mod $e$ 
using certain $p$-adic measures
arising in the study of Galois actions on paths on $\bold P^1-\{0,1,\infty\}$
(see also \cite{W4}).
The purpose of this paper is to complete the construction to include 
the case $p\mid m$ and to lift it over $\hat\Z=\varprojlim_N (\Z/N\Z)$.

Throughout this paper, we fix an embedding of $\overline{\Q}$
into $\C$. For any subfield $F\subset \C$, denote by $G_F$ the
absolute Galois group $\Gal(\bar F/F)$.

\begin{Theorem}  \label{mainthm1}
Let $m$ and $a$ be mutually prime integers with $m>1$, $0<a<m$.
Then, for every $\sigma\in G_{\Q(\mu_m)}$,
there exists a certain measure $\hat\zeta_{a,m}(\sigma)$ 
in $\hat\Z[[\hat\Z]]$ such that
for every prime $p$, its image $\hat\zeta_{p,a,m}(\sigma)$ 
in $\Z_p[[\Z_p]]$ has the following 
integration properties over $\Z_p^\times$:
\begin{equation*}
\int_{\Z_p^\times}b^{k-1} d\hat\zeta_{p,a,m}(\sigma)(b)
=
\begin{cases}
(1-\chi_p(\sigma)^k)\cdot m^{k-1}\cdot\frac{1}{k}B_k(\frac{a}{m}) &(p\mid m); \\
(1-\chi_p(\sigma)^k)\cdot m^{k-1}\left(
\frac{1}{k}B_k(\frac{a}{m})-\frac{p^{k-1}}{k}B_k(\frac{\la ap^{-1}\ra }{m})
\right) 
&(p\nmid m)
\end{cases}
\end{equation*}
for all integers $k\ge 1$, 
where  
$\chi_p:G_\Q\to\Z_p^\times$ denotes the $p$-adic cyclotomic character,
and $\la ap^{-1}\ra$ represents the least positive integer such that
$\la ap^{-1}\ra p\equiv a \mod m$.

\end{Theorem}

\begin{Remark}
Note that, in the above theorem, the case $m=1$ is excluded. 
In fact, the case $m=a=1$ corresponds to the $\hat\Z$-zeta function
treated in \cite{W2}.
This separation of treatment is necessary for the appearance of
tangential base point $\ovec{10}$ in the construction of measure, 
which causes replacements of  both
$B_k(\frac{a}{m})$, $B_k(\frac{\la ap^{-1} \ra}{m} )$ of RHS by $B_k(1)$. 
\end{Remark}

\begin{Remark} \label{smallnote1}
More generally, we construct the measure $\hat\zeta_{a,m}(\sigma)
\in \hat\Z[[\hat\Z]]$ for $m>1$ and $m\nmid a$ which satisfies
the above integration property for all primes $p|m$ with $p\nmid a$
(cf. Remark \ref{smallnote2}).
\end{Remark}

\begin{Remark}
\label{p-adicLbeta}
{}Using any $\sigma\in G_{\Q(\mu_m)}$ with $\chi_p(\sigma)^{e}\ne 1$, we obtain
from $\hat\zeta_{p,a,m}$ 
a set of $p$-adic Hurwitz functions 
$\{L_p^{[\beta]}(s,a,m)\}_{\beta\in(\Z/e\Z)}$
by the standard integral
$$
L_p^{[\beta]}(s;a,m)=
\frac{1}{1-\omega(\chi_p(\sigma))^\beta[\chi_p(\sigma)]^{1-s}}
\int_{\Z_p^\times}
[b]^{1-s}b^{-1}\omega(b)^\beta
d\hat\zeta_{p,a,m}(\sigma)(b)
$$
where $\omega:\Z_p^\times\to \mu_e$ is the Teichm\"uller character,
and for every $b\in \Z_p^\times$, $[b]\in 1+q\Z_p$ is defined 
by $b=[b]\omega(b)$.
Note that the above integral converges in $s\in\Z_p$ except when
it has a pole at $s=1$ in the case $\beta\equiv 0\pmod e$.
It follows from Theorem \ref{mainthm1} that, for each $\beta\in\Z/e\Z$,
the $L$-function $L_p^{[\beta]}(s;a,m)$ has the interpolation property:
\begin{equation}
L_p^{[\beta]}(1-k;a,m)=
\begin{cases}
\frac{m^{k-1}}{k}B_k(\frac{a}{m}) &(p\mid m); \\
\frac{m^{k-1}}{k}
\left(
B_k(\frac{a}{m})-p^{k-1}B_k(\frac{\la ap^{-1}\ra }{m})
\right) 
&(p\nmid m)
\end{cases}
\end{equation}
for all $k\ge 1$ with $k\equiv \beta\mod e$.
Since $\Z_{>0,\equiv \beta(\mathrm{mod}\, e)}$ is dense in the space 
$\beta+\frac{q}{p}\Z_p$ ($=\Z_p$ ($p>2$), $2\Z_2$ or $1+2\Z_2$),
the above interpolation property shows that 
$L_p^{[\beta]}(s,a,m)$ is determined independently of $\sigma$ (at least) on
that space.
In particular when $\beta\equiv 0\pmod e$ and $p>2$, 
$L_p^{[0]}(s;a,m)=-\zeta_p^{Sh}(s;a,m)$ for $s\in\Z_p-\{1\}$.
See also Appendix \ref{CohenApp} for relations of $L_p^{[\beta]}(s,a,m)$
with Cohen's Hurwitz zeta functions $\zeta_p(s,x)$.

In the present paper, we hope to make a small step towards the quest
of Coates about existence of zeta functions on $\hat\Z$ with values in $\hat\Z$
\cite[Introduction]{W2}.
\end{Remark}

The mapping 
$\hat\zeta_{a,m}$
in Theorem \ref{mainthm1} gives a 1-cocycle
$G_{\Q(\mu_m)}\to \hat\Z(1)[[\hat\Z(-1)]]$ whose 
$(k-1)$st moment integral gives rise to a cohomology class in 
$H^1(G_{\Q(\mu_m)},\hat\Z(k))$ for $k\ge 2$.
In fact, we will show in Corollary \ref{MomentIntOverZp}:
$$
\int_{\Z_p} b^{k-1}d
\hat\zeta_{p,a,m}(\sigma)(b)
=\frac{m^{k-1}}{k}{B_k\!\left(\frac{a}{m}\right)}
(1-\chi_p(\sigma)^k)
\quad (\sigma\in G_{\Q(\mu_m)},\ k\ge 2)
$$
which implies that the $p$-adic image of the 
above cohomology class is torsion with order 
calculated explicitly by Bernoulli values. 
It is noteworthy that 
this cohomology class is closely related to the 
$\xi_m^a$-component of the $\Z(k)$-torsor
`$P_{m,k}+(-1)^k\epsilon P_{m,k}$' over $\mu_m$
studied by Deligne in  
\cite[Proposition 3.14, Lemma 18.5]{De}.

\bigskip
{\it Acknowledgement}: 
This work was partially supported by JSPS KAKENHI Grant Number JP26287006. 
The authors would like to thank the referee for many
valuable suggestions including a crucial remark to simplify the proof of 
Lemma \ref{corevalue}.

\section{The Kummer-Heisenberg measure $\kk_1$ }


\subsection{Cyclic coverings}
\label{sec2.1}
Let $F\subset\C$ be a finite extension of $\Q$ with the
algebraic closure $\overline{F}\subset\C$.
For any (normal) algebraic variety $V$ over $F$ and $F$-rational points
$x,y\in V(F)$, we write
$\pi^\et_1(V;y,x)$ for the set of \'etale paths from $x$ to $y$
on the geometric variety $V\otimes \overline{F}$, and
$\pi^\et_1(V;x)=\pi^\et_1(V;x,x)$ for 
the \'etale fundamental group with base point $x$.
Denote by $\pi_1^\proP(V,x)$ the maximal pro-$p$ quotient of $\pi_1^\et(V,x)$, and 
by $\pi_1^\proP(V;y,x)$ the natural push forward of $\pi^\et_1(V;y,x)$ 
induced from the projection $\pi_1^\et(V,x)\twoheadrightarrow\pi_1^\proP(V,x)$.

For each $n\ge 1$, 
write $\xi_n:=\exp(\frac{2\pi i}{n})$ so that
$\mu_n:=\{1,\xi_n,\xi_n^2,\dots,\xi_n^{n-1}\}$.
Let  
$$
V_n:=\mathbf{P}^1\setminus\{0,\mu_n,\infty\},
$$ 
where we understand 
$\{0,\mu_n,\infty\}$ is the abbreviation of $\{0,\infty\}\cup \mu_n$.
Regard $V_n(\C)=\C^\times\setminus\mu_n$.
Let $\ovec{01}_n$ be the tangential base point on $V_n$ represented by
the unit tangent vector and denote for simplicity $\ovec{01}$.
Then, for each $n\ge 1$, there is a standard cyclic \'etale cover 
$p_n:V_{n}\to V_1$ given by $z\mapsto z^n$ which sends 
$\ovec{01}_n$ to a Galois functor equivalent to $\ovec{01}_1$ on $V_1$.
Thus, without ambiguity, we may omit the index of $\ovec{01}$ on $V_n$
and regard $(V_n,\ovec{01})$ as a
pointed \'etale cover over $(V_1,\ovec{01})$.
By standard Galois theory, it allows us to 
identify $\pi_1^\et(V_{n},\ovec{01})$ as a subgroup of 
$\pi_1^\et(V_1,\ovec{01})$.

Let $x,y$ be the generators of $\pi_1^\et(V_1,\ovec{01})$ given by 
the loops based at $\ovec{01}$ on $V_1=\mathbf{P}^1-\{0,1,\infty\}$
running around $0,1$ once anti-clockwise respectively. 
Then, it is easy to see that, as a subgroup of it, 
$\pi_1^\et(V_{n},\ovec{01})$ is freely generated by 
$x_n:=x^n$ and $y_{b,n}:=x^{-b}yx^b$ ($0\le b<n$).


\subsection{Galois associators and Kummer-Heisenberg measure} 
\label{sec.KH}

Now, let 
$z$ be an $F$-point of $V_1=\mathbf{P}^1-\{0,1,\infty\}$. 
We have the canonical comparison map
$$
\pi_1(V_1(\C);z,\ovec{01})\longrightarrow
\pi_1^{\text{\'et}}(V_1;z,\ovec{01})
$$
from the set of homotopy classes of paths from 
$\ovec{01}$ to $z$ on $V_1(\C)$ to the \'etale
paths from $\ovec{01}$ to $z$ on $V_1\otimes \bar F$.
The Galois group $G_F$ acts on the profinite group
$\pi_1^\et(V_1,\ovec{01})$ and its torsor of paths
$\pi_1^{\text{\'et}}(V_1;z,\ovec{01})$. 


Let us fix an \'etale path $\gamma\in\pi^\et_1(V_1(\C);z,\ovec{01})$.
For $\sigma\in G_F$, define the {\it Galois associator} for
the path $\gamma$ by
\begin{equation} \label{GalAssociator}
\wf_\gamma(\sigma):=\gamma^{-1}\cdot\sigma(\gamma)
\in\pi_1^{\text{\'et}}(V_1,\ovec{01}),
\end{equation}
where 
$\sigma(\gamma):=\sigma\circ\gamma\circ\sigma^{-1}$.

Write $\pi'$ for the commutator subgroup of a profinite group $\pi$.
The abelianization of 
$\wf_\gamma(\sigma)$ is known (cf. \cite[Proposition 1]{NW1})
to be expressed as:
\begin{equation} \label{AbelianizationAssociator}
\wf_\gamma(\sigma)
\equiv
x^{\rho_{z,\gamma}(\sigma)}
y^{\rho_{1-z,\gamma}(\sigma)}
\mod \pi_1^{\text{\'et}}(V_1,\ovec{01})' ,
\end{equation}
with the $\hat\Z$-valued functions 
$$
\rho_{z,\gamma},\,  \rho_{1-z,\gamma}
:G_F\to \hat\Z
$$
the Kummer 1-cocycles associated with the roots of 
$z$ and $1-z$.
They are respectively calculated along $\gamma$ with the 
above chosen base of the Tate module
\begin{equation} \label{Tate_module}
(\xi_n)_{n\ge 1}\in\hat\Z(1):=\varprojlim_n \mu_n .
\end{equation}
For the latter 
$\rho_{1-z,\gamma}$, we understand the points 
$\ovec{01}$ and $1-z$ are 
connected by the unit segment $[0,1]$ on $\mathbf{P}^1$ 
followed with 
the reversed path of 
$\gamma$ by $(\ast\mapsto 1-\ast)$.
We sometimes omit the mention to $\gamma$ when it is 
obvious from context.

\begin{Definition}
\label{DefKappaValues}
Let $\sigma\in G_F$ and 
set
$$
\wf^\flat_\gamma(\sigma):=
x^{-\rho_{z,\gamma}(\sigma)} \wf_\gamma(\sigma) \quad
(\sigma\in G_F).
$$
which belongs to the subgroup
$\pi_1(V_n,\ovec{01})\subset \pi_1(V_1,\ovec{01})
$ by (\ref{AbelianizationAssociator}) for every $n\ge 1$.
Given $0\le b<n$, we define
$\kappa_{z,\gamma}^{(n)}(\sigma)(b)\in\hat\Z$
by the congruence
$$
\wf^\flat_\gamma(\sigma)\equiv
\prod_{b=0}^{n-1}
{y_{b,n}}^{\kappa_{z,\gamma}^{(n)}(\sigma)(b)}
$$
modulo 
$\pi_1^\et(V_n,\ovec{01})'$: 
the commutator subgroup of 
$\pi_1^\et(V_n,\ovec{01})$.
\end{Definition}

\begin{Proposition}[See \cite{NW1} Lemma 1]
For each $\sigma\in G_F$, the system of functions
$$
\left\{
\Z/n\Z\ni b \mapsto 
\kappa_{z,\gamma}^{(n)}(\sigma)(b)\in\hat\Z
\right\}_{n\in\N}
$$
running over $n\ge 1$ defines 
a $\hat\Z$-valued measure on $\hat\Z$.
\qed
\end{Proposition}

We shall denote the above measure by
$$
\kk_1(\gamma:\ovec{01}\dashto z)(\sigma)
\text{ or }
\kk_1(z)_\gamma(\sigma)
$$
and call it the {\it Kummer-Heisenberg measure} 
associated
with the path $\gamma:\ovec{01}\dashto z$.
We view it as an element of the Iwasawa algebra
$\hat\Z[[\hat\Z]]$.
Recall that $\hat\Z(1)$ in 
(\ref{Tate_module}) 
is the Galois module 
$\hat\Z$ acted on by $G_F$ by
multiplication by the cyclotomic character. 
Let $\hat\Z(-1)$ be its dual.

%

\begin{Proposition}
The function
$$
\kk_1(\gamma:\ovec{01}\dashto z):
G_F\to \hat\Z(1)[[\hat\Z(-1)]]
$$
is a cocycle. Namely it holds that
$$
\kappa_{z,\gamma}^{(n)}(\sigma\tau)(b)
=
\kappa_{z,\gamma}^{(n)}(\sigma)(b)
+\chi(\sigma)\cdot
\kappa_{z,\gamma}^{(n)}(\tau)(\chi(\sigma)^{-1} b)
$$
for $\sigma,\tau\in G_F$, $n\ge 1$, $b\in\Z/n\Z$.
\end{Proposition}

\begin{proof}
By the definition of $f_\gamma$ (\ref{GalAssociator}), we have
$f_\gamma(\sigma\tau)=f_\gamma(\sigma)\cdot \sigma(f_\gamma(\tau))$,
hence 
$\wf^\flat_\gamma(\sigma\tau)
\equiv \wf^\flat_\gamma(\sigma)\cdot \sigma(\wf^\flat_\gamma(\tau))$
modulo $\pi_1(V_n)'$.
The assertion follows from this and the observation
$$
\sigma(y_{b,n})\equiv
x^{-\chi(\sigma) b}
y^{\chi(\sigma)} x^{\chi(\sigma) b}
\equiv
(y_{\chi(\sigma)b,n})^{\chi(\sigma)}
$$
modulo $\pi_1(V_n)'$.
\end{proof}

\begin{Remark}
\label{remark.KH}
In \cite[Lemma 1]{NW1}, we introduced a compatible sequence 
$(\kappa_n)_n$ in the projective system
$\varprojlim_n\hat\Z[\Z/n\Z]$
which forms a measure $\hat\kappa\in\hat\Z[[\hat\Z]]$.
We call $\hat\kappa$ (resp. $\kk_1$) the Kummer-Heisenberg
measure in $e$-form (resp. $t$-form) in the terminology 
of Appendix \ref{PathApp}. 
These two measures are `oppositely directed' 
mainly because of different choice of path conventions as follows.
After identification $\hat\Z\isom \hat\Z(1)$ by 
$1\mapsto (\xi_n=\exp(2\pi i/n))_n$, let 
$\epsilon$ denote the involution on $\hat\Z(1)$ induced by
$\xi\mapsto \xi^{-1}$. Then, we have
$\kk_1(\sigma)=\epsilon\cdot\hat\kappa(\sigma)$
$(\sigma\in G_F)$ 
as elements of $\hat\Z[[\hat\Z(1)]]$.
\end{Remark}


\section{Adelic Hurwitz measure}

\subsection{Paths to roots of unity}  \label{sec3.1}

Fix $m\in\N_{>1}$ and let $a$ be an integer with $m\nmid a$.
Let $\iota:V_1\to V_1$ be the involution given by 
$\iota(z)=z^{-1}$. 
For any path $\gamma$ on $V_1(\C)$ from $\ovec{01}$
to $\xi_m^a$, we set another path $\bar\gamma$ from
$\ovec{01}$ to $\xi_m^{-a}$ by
$$
\bar\gamma:=\iota(\gamma)\cdot\Gamma_\infty
$$
where $\Gamma_\infty$ is a path on $V_1(\C)$ from $\ovec{01}$
to $\ovec{\infty 1}$ as in Figure 1.

\begin{figure}[H]
\begin{center}
\includegraphics[width=5cm, bb=0 0 393 106]{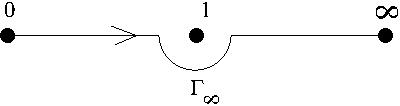}
\end{center}
\caption{$\Gamma_\infty$ is a path from $\protect \overrightarrow{01}$ 
to $\protect \overrightarrow{\infty 1}$ }
\end{figure}

\noindent
Write $\frac{a}{m}=\left\lfloor \frac{a}{m}\right\rfloor +\{\frac{a}{m}\}$
so that $0\le\{\frac{a}{m}\}<1$,
and define the path
$\Gamma_{a/m}:\ovec{01}\dashto \xi_m^a$ to be the composition
$\Gamma_{\{a/m\}}\cdot x^{\lfloor a/m\rfloor}$, where 
$\Gamma_{\{a/m\}}$ is the path
%
illustrated as in Figure 2.

\vspace{-8mm}
\begin{figure}[H]
\begin{center}
\includegraphics[width=5cm, bb=0 0 372 352]{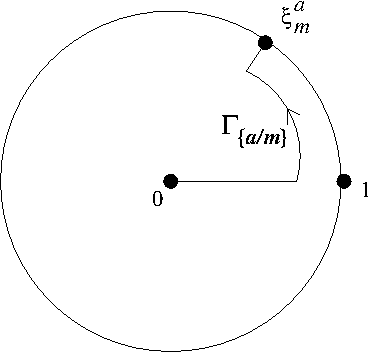}
\end{center}
\caption{$\Gamma_{\{a/m\}}$ is a path from $\protect \overrightarrow{01}$ 
to $\xi_m^a$}
\end{figure}

It is easy to see the following lemma.
\begin{Lemma}
\label{LemmaC1}
Along the above paths 
$\Gamma_{a/m}:\ovec{01}\dashto \xi_m^a$ 
and 
$\bar\Gamma_{a/m}:\ovec{01}\dashto \xi_m^{-a}$,
the associated 
Kummer 1-cocycles are coboundaries satisfying
\begin{equation*}
\rho_{\xi_m^a, \Gamma_{a/m}}(\sigma)=
\frac{a}{m}
(\chi(\sigma)-1),\quad
\rho_{\xi_m^{-a}, \bar\Gamma_{a/m}}(\sigma)=
-\frac{a}{m}
(\chi(\sigma)-1)
\end{equation*}
for $\sigma\in G_{\Q(\mu_m)}$. 
%
\end{Lemma}

\begin{proof}
The first formula is immediate from the definition and the 
identification $\hat\Z\cong\hat\Z(1)$ by $1\mapsto (\xi_n)_{n\ge 1}$.
For the second, it suffices to note that the image of
$\bar\Gamma_{a/m}$ by 
$\mathbf{P}^1-\{0,1,\infty\}\hookrightarrow \mathbf{P}^1-\{0,\infty\}$
is topologically homotopic to the complex conjugate of $\Gamma_{a/m}$.
\end{proof}

\begin{Remark}
\label{rem3.2}
It is worth noting that $\Gamma_{\alpha}x^{n}=\Gamma_{\alpha+n}$ for any $\alpha\in\Q$
and $n\in\Z$. This additivity property does not hold for $\bar\Gamma_\alpha$ in general. 
Still, if $0\le \alpha\le 1$, then it holds that
$\bar\Gamma_{\alpha}=\Gamma_{-\alpha}=\Gamma_{n-\alpha}x^{-n}$ for every $n\in\Z$.
This last point will play a crucial role later in Lemma \ref{lem5.6}.
\end{Remark}


\subsection{Translation of a measure}
\label{sec.TM}

Let $p_n:V_{n}\to V_1$ be the cyclic cover of degree $n$
considered in \S \ref{sec2.1}.
For an \'etale path $\gamma:\ovec{01}\dashto z$ on $V_1$,
we shall write 
$$
\gamma_n\bigl(=(\gamma)_n\bigr):\ovec{01}_n\dashto z^{1/n}
$$
to denote the lift of $\gamma$ to $V_n$ for all $n\ge 1$.
Let us fix $\sigma\in G_F$. 
Note that the end point  $z^{1/n}$ may or may not be fixed by the $\sigma$.
By (\ref{AbelianizationAssociator}), we have
\begin{equation} \label{wf_gamma_flat}
\wf^\flat_\gamma(\sigma)=
(\gamma\cdot x^{\rho_{z,\gamma}(\sigma)})^{-1} \sigma(\gamma)
\in\pi_1^{\text{\'et}}(V_1,\ovec{01})' .
\end{equation}
But since $V_n$ is an abelian cover of $V_1$, 
$\pi_1^{\text{\'et}}(V_1,\ovec{01})'$ is contained in $\pi_1(V_n,\ovec{01})$.
Therefore (\ref{wf_gamma_flat}) implies that
the lifts of $\gamma\cdot x^{\rho_{z,\gamma}(\sigma)}$ and of $\sigma(\gamma)$
departing at $\ovec{01}_n$ on $V_n$ end at the same point
$\sigma(z^{1/n})=\xi_n^{\rho_{z,\gamma}(\sigma)} z^{1/n}$.
Since the lift $(x^{\rho_{z,\gamma}(\sigma)})_n$
of $x^{\rho_{z,\gamma}(\sigma)}$ from $\ovec{01}_n$ ends at 
$\xi_n^{\rho_{z,\gamma}(\sigma)}\ovec{01}_n$, the subsequent 
path $\gamma_{n,\sigma}$ should lift $\gamma$ so as to start 
from that point $\xi_n^{\rho_{z,\gamma}(\sigma)}\ovec{01}_n$ 
with ending at the point $\sigma(z^{1/n})$ on $V_n$:
\begin{eqnarray}\label{liftingpath}
\xymatrix@C=3pc{
\ovec{01}_n  \ar@{-->}[r]^{x^{\rho_{z,\gamma}(\sigma)}\quad} 
\ar@/_2pc/@{-->}[rr]_{(\sigma(\gamma))_n}
&  
\xi_n^{\rho_{z,\gamma}(\sigma)} \ovec{01}_n
\ar@{-->}[r]^{\gamma_{n,\sigma}}
&
\sigma(z^{1/n})
}.
\end{eqnarray}
In summary, 
writing $(\sigma(\gamma))_n$ for the lift of $\sigma(\gamma)$ 
from $\ovec{01}_n$ on $V_n$,
we may express $\wf^\flat_\gamma(\sigma)$ as 
the composition of those three paths 
$$
\wf^\flat_\gamma(\sigma)=
(x^{\rho_{z,\gamma}(\sigma)})_n^{-1} \cdot
\gamma_{n,\sigma}^{-1}\cdot
(\sigma(\gamma))_n
$$
on $V_n$.

Below, we shall see magnification of the base space $\hat\Z$ on a coset 
$s+r\hat\Z$ ($s,r\in\Z$, $r\ge 1$) 
under the measure $\kk_1(z)_\gamma(\sigma)$
can be interpreted 
as a twisted lifting of the reference path 
$\gamma:\ovec{01}\dashto z$
to $V_r$ followed with `$s$-rotated' 
embedding by $V_r\hookrightarrow V_1$.

Set an `$s$-modified' path $\gamma_{\la -s \ra}:\ovec{01}\dashto z$,
for the given path $\gamma:\ovec{01}\dashto z$ on $V_1$, 
by
\begin{equation}
\label{gamma-s}
\gamma_{\la -s \ra}:=\gamma\cdot x^{-s}.
\end{equation}
It follows easily that
\begin{equation}
\label{rho_z_for_gamma-s}
\rho_{z,\gamma_{\la -s \ra}}(\sigma)=\rho_{z,\gamma}(\sigma)-s(\chi(\sigma)-1) 
\quad(\sigma\in G_F).
\end{equation}
Suppose that $\xi_r, z^{1/r}\in F$.
Write 
\begin{align*}
\gamma_{\,r}=(\gamma)_r&:\ovec{01}_r\dashto\ z^{1/r}, \\
\gamma_{\la -s \ra,r}=(\gamma_{\la -s \ra})_r  &:\ovec{01}_r\dashto\ \xi_r^{-s}\, z^{1/r}
\end{align*}
for the lifts 
of the paths $\gamma$ and $\gamma_{\la -s \ra}$ by $p_r:V_r\to V_1$ respectively,
and 
\begin{align*}
\gamma_{\,r\ast}=\jmath_r(\gamma_r)&:\ovec{01}\dashto\ z^{1/r}, \\
\gamma_{\la -s \ra,r\ast}=\jmath_r(\gamma_{\la -s \ra,r} ) &:\ovec{01} \dashto\ \xi_r^{-s}\, z^{1/r}
\end{align*}
for the images of paths
$\gamma_{\,r}$, $\gamma_{\la -s \ra,r}$ on $V_r$ by the 
immersion $\jmath_r:(V_r,\ovec{01}_r)\hookrightarrow (V_1,\ovec{01})$ respectively. 
It follows that
\begin{align}
\label{lifted_rho_z}
\rho_{z^{1/r},{\gamma_{\la -s \ra,r\ast}}}(\sigma)
&=\rho_{z^{1/r},(\gamma)_{r\ast}}(\sigma)-\frac{s}{r}(\chi(\sigma)-1) \\
&=\frac{1}{r}\rho_{z,\gamma}(\sigma)-\frac{s}{r}(\chi(\sigma)-1) 
 \notag
\end{align}
for every $\sigma\in G_F$.

\begin{Lemma}
\label{magnification}
Notations being as above, with assumptions $\xi_r, z^{1/r}\in F$ and $\sigma\in G_F$.

(i)
For every $n\ge 1$, it holds that
$$
\kappa_{z,\gamma}^{(nr)}(\sigma)(vr+s\chi(\sigma))
=\kappa_{\xi_r^{-s}z^{1/r},{\gamma_{\la -s \ra,r\ast}}}^{(n)}(\sigma)(v)
\qquad (v=0,\dots,n-1),
$$
where $vr+s\chi(\sigma)$ in LHS is regarded $\in (\Z/nr\Z)$.

(ii)
For any continuous function $\varphi$ on $\hat\Z$, we have
$$
\int_{s\chi(\sigma)+r\hat\Z} 
\varphi(b) \, 
 d\kk_1(\gamma:\ovec{01}\dashto z)(\sigma)(b)
=
\int_{\hat\Z} \varphi(rv+s\chi(\sigma)) \, 
d\kk_1({\gamma_{\la -s \ra,r\ast}}:\ovec{01}
\dashto \xi_r^{-s}z^{1/r})(\sigma)(v).
$$
\end{Lemma}

\begin{proof}
In this proof, for $n\ge 1$, we denote $\pi(n):=\pi_1^\et(V_n,\ovec{01})$
and write 
\begin{equation} \label{defvarpi}
\varpi_{nr}:\pi(nr)\twoheadrightarrow \pi(n)
\end{equation}
for 
the surjection induced from the open immersion
$(V_{nr},\ovec{01}_{nr})\hookrightarrow (V_n,\ovec{01}_n)$.
Note that, among the standard generators $x^{nr}$, $y_{b,nr}$ ($b=0,\dots,nr-1$) of $\pi(nr)$, 
only $x^{nr}$ and $y_{vr,nr}$ ($v=0,\dots, n-1$) survive via $\varpi_{nr}$
to be $x^{n}$, $y_{v,n}$ ($v=0,\dots n-1$) in $\pi(n)$.

Noting that $x^{-u}yx^u=y_{u,nr}\equiv y_{u+nrk,nr}$ mod $\pi(nr)'$
for $u,k\in\Z$, we see from Definition \ref{DefKappaValues} 
that
\begin{align}
\label{LHSofLem3.2}
x^{s\chi(\sigma)}\cdot
\wf^\flat_\gamma(\sigma)\cdot
x^{-s\chi(\sigma)}
&\equiv \prod_{u=0}^{nr-1} {y_{-s\chi(\sigma)+u,nr}}^{\kappa_{z,\gamma}^{(nr)}(\sigma)(u)} 
\mod \pi(nr)'
\\
&\equiv \prod_{v=0}^{nr-1} {y_{v,nr}}^{\kappa_{z,\gamma}^{(nr)}(\sigma)(v+s\chi(\sigma))} 
\mod \pi(nr)'
\notag 
\end{align}
which should map via $\varpi_{nr}$ to the product over those $v$ multiples of $r$:
\begin{equation}  
\label{eq3.6}
\varpi_{nr}\left(
x^{s\chi(\sigma)}\cdot
\wf^\flat_\gamma(\sigma)\cdot
x^{-s\chi(\sigma)}\right)
\equiv
\prod_{v=0}^{n-1} {y_{v,n}}^{\kappa_{z,\gamma}^{(nr)}(\sigma)(rv+s\chi(\sigma))} 
\mod \pi(n)' 
\end{equation}
as $\pi(n)'\supset \varpi_{nr}(\pi(nr)')$.
We shall interpret the LHS of the above expression 
(\ref{LHSofLem3.2})
by applying the composition diagram 
(\ref{liftingpath}) to the path $\gamma_{\la -s \ra}:\ovec{01}\dashto z$ 
(\ref{gamma-s}) on $V_1$ and its lift $(\gamma_{\la -s \ra})_r=\gamma_{\la -s \ra,r}$ on $V_r$:
\begin{eqnarray*}
\xymatrix{
{} &V_{nr} 
\ar@{^{(}-_{>}}[r] \ar@{->>}[d] & 
V_n 
\ar@{->>}[d] & {}
    \\
[\gamma_{\la -s \ra,r}:\ovec{01}\dashto \xi_r^{-s} z^{1/r}]
\ar@{^{(}..}[r] & V_r 
\ar@{^{(}-_{>}}[r] \ar@{->>}[d] &
V_1 & [{\gamma_{\la -s \ra,r\ast}}]\ar@{_{(}..}[l] 
\\
\quad\qquad [\gamma_{\la -s \ra}:\ovec{01}\dashto z]
\ar@{^{(}..}[r] 
 &V_1 & {} & {\qquad .}
}
\end{eqnarray*}
We first derive:
\begin{align}
\label{LHSofLem3.2Second}
x^{s\chi(\sigma)}\cdot
\wf^\flat_\gamma(\sigma)\cdot
x^{-s\chi(\sigma)}
&=x^{s\chi(\sigma)}\cdot x^{-\rho_{z,\gamma}(\sigma)}\gamma^{-1}\sigma(\gamma)
\cdot
x^{-s\chi(\sigma)} \\
&=x^{s(\chi(\sigma)-1)-\rho_{z,\gamma}(\sigma)}\cdot
(\gamma x^{-s})^{-1}
\sigma(\gamma x^{-s})  \notag \\
&=(\gamma_{\la -s \ra}\cdot x^{\rho_{z,\gamma_{\la -s \ra}}(\sigma)})^{-1}\cdot
\sigma(\gamma_{\la -s \ra}). \notag
\end{align}
By (\ref{rho_z_for_gamma-s}), the former factor of path composition 
reads on $V_r$ 
\begin{equation*}
(\gamma_{\la -s \ra})_{r,\sigma} \cdot 
\bigl(x^{\rho_{z,\gamma_{\la -s \ra}}(\sigma)}\bigr)_r
=
(\gamma x^{-s})_{r,\sigma} \cdot 
\bigl(x^{\rho_{z,\gamma}(\sigma)-s(\chi(\sigma)-1)}\bigr)_r \\
\end{equation*}
where $(\gamma_{\la -s \ra})_{r,\sigma}$ stands for 
a suitable lift of $\gamma_{\la -s \ra}$ on
$V_r$, which arrives at
the same end point on $V_r$ 
as the latter $\sigma$-transformed factor 
$$
(\sigma(\gamma_{\la -s \ra}))_r:\ovec{01}_r\dashto \sigma(\xi_r^{-s}\,z^{1/r}).
$$
It turns out that $(\gamma_{\la -s \ra})_{r,\sigma}$ starts at 
$\xi_r^{\rho_{z,\gamma}(\sigma)-s(\chi(\sigma)-1)}\cdot\ovec{01}_r$
which is equal to $\ovec{01}_r$ by our assumption $\xi_r,z^{1/r}\in F$.
Thus we conclude 
\begin{equation}
(\gamma_{\la -s \ra})_{r,\sigma}
=\gamma_{\la -s \ra,r}
\left(= (\gamma\cdot x^{-s})_r\right).
\end{equation}
By virtue of this and (\ref{lifted_rho_z}), applying to (\ref{LHSofLem3.2Second})
the surjection $\varpi_r:\pi(r)\twoheadrightarrow\pi(1)$ determined by
$x^r\mapsto x$, $y_0\mapsto y$ and $y_1,\dots,y_{r-1}\mapsto 1$ as 
the case $n=1$ of (\ref{defvarpi}), we obtain
\begin{align*}
\varpi_r\left(
(x^{-\rho_{z,\gamma_{\la -s \ra}}(\sigma)})_r 
\cdot
(\gamma_{\la -s \ra})_{r,\sigma}^{-1}\cdot
\sigma(\gamma_{\la -s \ra,r})\right)  
&=
\varpi_r\left(
(x^r)^{-\rho_{z^{1/r},\gamma_r}(\sigma)+\frac{s}{r}(\chi(\sigma)-1)}
\gamma_{\la -s \ra,r}^{-1}\cdot
\sigma(\gamma_{\la -s \ra,r})\right) \\
&=
x^{-\rho_{z^{1/r},{\gamma_{\la -s \ra,r\ast}}}(\sigma)}\cdot
{\gamma_{\la -s \ra,r\ast}}^{-1}\cdot
\sigma({\gamma_{\la -s \ra,r\ast}})
\\
&=
\wf^\flat_{{\gamma_{\la -s \ra,r\ast}}}(\sigma)
\\
&\equiv
\prod_{v=0}^{n-1} {y_{v,n}}^{\kappa_{\xi_r^{-s}z^{1/r},{\gamma_{\la -s \ra,r\ast}}}^{(n)}(\sigma)(v)}
\mod \pi(n)' .
\end{align*}
This, combined with (\ref{eq3.6}) and (\ref{LHSofLem3.2Second}) and the compatibility
$\varpi_{nr}=\varpi_{r}|_{\pi(nr)}$, proves (i).
The assertion (ii) is just a formal consequence of (i).
\end{proof}

Suppose we are given a measure $\mu\in\hat\Z[[\hat\Z]]$. 
Let $m,a\in\Z$ be integers as in \S \ref{sec3.1} and pick $\nu\in\hat\Z^\times$.
Consider the coset $Q_{a\nu,m}:=\frac{a\nu}{m}+\hat\Z$ of 
$\hat\Z$ in $\Q_f=\hat\Z\otimes\Q$.
Then, obviously, 
$$
R_{a\nu,m}:=m\cdot Q_{a\nu,m}=a\nu+m\hat\Z \subset \hat\Z.
$$
We define the measure $[m,a\nu]_\ast(\mu)$ on $R_{a\nu,m}$ by
assigning to each open subset $U\subset R_{a\nu,m}$ the value
$\mu(U')$, where $U'$ is the inverse image of $U$ by the affine 
map $t\mapsto mt+a\nu $ ($t\in\hat\Z$).
Note that Lemma \ref{magnification} (ii) reads:
\begin{equation}
\label{magnifiedKappa}
\kk_1(\gamma:\ovec{01}\dashto z)(\sigma)
|_{s\chi(\sigma)+m\hat\Z}
=
[m,s\chi(\sigma)]_\ast\Bigl(
\kk_1({\gamma_{\la -s \ra,m\ast}}:\ovec{01}
\dashto \xi_m^{-s}z^{1/m})(\sigma)
\Bigr)
\end{equation}
for $\sigma\in G_F$, $m,s\in\Z$, $m\ge 1$, where $\ast|_{s\chi(\sigma)+m\hat\Z}$
in LHS designates the restricted measure on 
$R_{s\chi(\sigma),m}=s\chi(\sigma)+m\hat\Z\subset \hat\Z$.

Let $\iota$ denote the action of the
complex conjugation on $\hat\Z(1)[[\hat\Z(-1)]]$, that is,
the action of $-1\in\hat\Z^\times$.
It is straightforward to see 
\begin{equation}
\label{iota-translation}
\iota\circ [m,a\nu]_\ast=[m,-a\nu]_\ast\circ \iota.
\end{equation}

Now we are ready to introduce the fundamental object
of our study.
Let $m>1$ and $a\in\Z$ as above, and 
let $\Gamma_{a/m}\in\pi(V_1(\C);\xi_m^a,\ovec{01})$ 
be the path introduced in \S \ref{sec3.1}.
%

\begin{Definition}[{\bf $\hat\Z$-Hurwitz and adelic Hurwitz measure}] 
\label{defhatzeta}
For each $\sigma\in G_{\Q(\mu_m)}$ we define the
$\hat\Z$-Hurwitz measure 
$\boldzeta_{a/m}(\sigma)\in\hat\Z[[\hat\Z]]$
and the adelic Hurwitz measure 
${\hat\zeta}_{a,m}(\sigma)\in\hat\Z[[\hat\Z]]$
by the formulas
\begin{align*}
\boldzeta_{a/m}(\sigma)
&:=
\kk_1\bigl(
\pathtoD{\ovec{01}}{\bar\Gamma_{a/m}}{\xi_m^{-a}}
\bigr)(\sigma)
+\iota\left(
\kk_1\bigl(
\pathtoD{\ovec{01}}{\Gamma_{a/m}}{\xi_m^{a}}
\bigr)(\sigma)
\right)
; 
\\
{\hat\zeta}_{a,m}(\sigma)
&:=
[m,a\chi(\sigma)]_\ast
\, \boldzeta_{a/m}(\sigma) .
\end{align*}
\end{Definition}

\section{Geometrical interpretation of translation of measure}

In this section we address the fact
that translation of Kummer-Heisenberg measure by 
$[m, a\chi]_\ast$ corresponds to path composition with 
the loop $x^\alpha$. 
We work however only with 
$p$-adic measures.

\subsection{$p$-adic Galois polylogarithms}
Let $\gamma$ be an \'etale path on $V_1$ from $\ovec{01}$
to an $F$-rational (possibly tangential) point $z$.
Let $\Q_p\lala X,Y\rara$ be the 
non-commutative power series ring in two variables $X$, $Y$,
and write $\mathscr{E}:\pi_1^\proP(V_1,\ovec{01})\hookrightarrow
\Q_p\lala X,Y\rara$ for the embedding that
sends the standard generators $x,y$ to $\exp(X)$,
$\exp(Y)$.
We define $I_Y$ to be the ideal of $\Q_p\lala X,Y\rara$
generated by monomials containing $Y$ twice or more.

For $\sigma\in G_F$, set
\begin{align*}
\Lambda_\gamma(\sigma)&:=\mathscr{E}(\wf_\gamma(\sigma));   \\
\overline{\Lambda_\gamma}(\sigma)&:=\mathscr{E}(\wf^\flat_\gamma(\sigma))
=\exp(-\rho_{z,\gamma}(\sigma) X)\cdot\Lambda_\gamma(\sigma).
\end{align*}

\begin{Definition} \label{Def4.1}
Define $p$-adic Galois polylogarithms
$\ell i_k(z)_\gamma, \Li_k(z)_\gamma:
G_F\to \Q_p$ by the congruence expansion
\begin{align*}
\log\Lambda_\gamma(\sigma)&\equiv
\rho_{z,\gamma}(\sigma)+\sum_{k=1}^\infty
(-1)^{k-1} \ell i_k(z)_\gamma(\sigma)\,(\mathrm{ad} X)^{k-1}(Y),
\\
\log \overline{\Lambda_\gamma}(\sigma)&\equiv
\sum_{k=1}^\infty
(-1)^{k-1} \Li_k(z)_\gamma(\sigma)\,(\mathrm{ad} X)^{k-1}(Y)
\end{align*}
modulo the ideal $I_Y$. 
\end{Definition}

\begin{Proposition} \label{IdLili}
The family of functions
$\{\rho_{z,\gamma}$, $\ell i_k(z)_\gamma$, 
$\Li_k(z)_\gamma:G_F\to\Q_p\}_{k\ge 1}$
satisfy
\begin{align*}
\Li_k(z)_\gamma&=\sum_{i=1}^k
\frac{\rho_{z,\gamma}^{k-i}}{(k+1-i)!}\,
\ell i_i(z)_\gamma, 
\tag{i}
\\
\ell i_k(z)_\gamma &=\sum_{s=0}^{k-1}
\frac{B_s}{s!} \rho_{z,\gamma}^s\,
\Li_{k-s}(z)_\gamma
\tag{ii}
\end{align*}
for $k=1,2,\dots$.
\end{Proposition}

In fact, this proposition is a formal consequence of the following lemma:

\begin{Lemma}
Let $K$ be a field of characteristic zero, and suppose that
two sequences $\{b_i\}_{i\ge 0}$, $\{u_i\}_{i\ge 0}$ in $K$
satisfy the congruence
$$
e^{-u_0X}e^{u_0X+\sum_{k=0}^\infty u_{k+1} (\mathrm{ad} X)^k (Y)}
\equiv e^{b_0X+\sum_{k=0}^\infty b_{k+1} (\mathrm{ad} X)^k (Y)}
\mod I_Y
$$
as non-commutative power series in $X,Y$.
Then, $b_0=0$ and, for $k=1,2,\dots$,
\begin{equation*}
b_k=\sum_{i=1}^k  \frac{(-u_0)^{k-i}}{(k+1-i)!} u_i,\qquad 
u_{k}=\sum_{s=0}^{k-1} \frac{B_s}{s!}(-u_0)^s b_{k-s},
\end{equation*} 
where $B_0, B_1,\dots$ are Bernoulli numbers defined by 
$\sum_{s=0}^\infty \frac{B_s}{s!}T^s=\frac{T}{e^T-1}$.
\end{Lemma}

\begin{proof}
We use the classical Campbell-Hausdorff formula
$$
\log(e^\alpha e^\beta)\equiv \beta+\sum_{n=0}^\infty
\frac{B_n}{n!}(\mathrm{ad} \beta)^n(\alpha)
\mod \mathrm{deg}(\alpha)\ge 2.
$$
Set
$-\alpha=\sum_{i=0}^\infty b_{i+1}(\mathrm{ad} X)^i (Y)$
and $-\beta=u_0X$ so that congruences 
mod deg($\alpha)\ge 2$ derive those mod $I_Y$.
It follows that 
$\log(e^{u_0X}e^{b_1Y+b_2[X,Y]+\dots})$ is congruent to 
$u_0X+\sum_{k=0}^\infty\bigl(
\sum_{s=0}^k\frac{B_s}{s!}(-u_0)^sb_{k+1-s}
\bigr)
(\mathrm{ad} X)^k(Y)$ 
mod $I_Y$. 
This is equivalent to the equality
\begin{equation}
\sum_{k=0}^\infty u_{k+1}T^k=
\left(\frac{-u_0T}{e^{-u_0T}-1}\right) \sum_{k=0}^\infty b_{k+1}T^k.
\end{equation}
The assertion follows from this immediately. 
\end{proof}

\subsection{Extension for $\Q_p$-paths}
Let $\pi_{\Q_p}(\ovec{01})$ be the pro-algebraic hull of 
the image of the above embedding
$\mathscr{E}:\pi_1^\proP(V_1,\ovec{01})\hookrightarrow
\Q_p\lala X,Y\rara$, and extend it to the inclusion of 
path torsors 
$\pi_1^\proP(V_1;z,\ovec{01})\hookrightarrow 
\pi_{\Q_p}(z,\ovec{01})$ naturally. 
The elements of $\pi_{\Q_p}(\ovec{01})$, 
$\pi_{\Q_p}(z,\ovec{01})$ will be simply called $\Q_p$-paths,
and the action of the Galois group $G_F$ on the pro-$p$ paths
extends to that on the $\Q_p$-paths in the obvious manner.

For each $\Q_p$-path $\gamma:\ovec{01}\dashto z$ and $\sigma\in G_F$,
we may define the Galois associator
$f_\gamma(\sigma):=\gamma^{-1}\cdot\sigma(\gamma)\in \pi_{\Q_p}(\ovec{01})$
extending (\ref{GalAssociator}).
Then, define $\rho_{z,\gamma}, li_k(z,\gamma): G_F\to \Q_p$ 
$(k=1,2,\dots)$ as 
the coefficients in $\log(f_\gamma(\sigma))$ so as to extend 
the congruence in Definition \ref{Def4.1} mod $I_Y$, and then, define
$\Li_k(z)_\gamma:G_F\to \Q_p$ $(k=1,2,\dots)$ 
as the coefficients of 
$\log(\exp(-\rho_{z,\gamma}(\sigma)X)\cdot f_\gamma(\sigma))$
again as the extension of Definition \ref{Def4.1}.
Then, it is simple to see that the identities in Proposition \ref{IdLili} 
hold true for 
$\Q_p$-paths $\gamma:\ovec{01}\dashto z$ in the same forms.

\subsection{Relation with $\kk_{1,p}$}
We now arrive at the stage to connect the $\ell$-adic 
polylogarithms $\Li_k$ and the Kummer-Heisenberg measure
$\kk_1$.
In \cite{NW1}, we showed that, for pro-$p$ paths $\gamma:\ovec{01}\dashto z$,
the function $\Li_k(z)_\gamma$ 
multiplied by $(k-1)!$ can be written 
by a certain polylogarithmic character $\tilchi_k(z)_\gamma:G_F\to \Z_p$
defined by 
Galois transformations of certain sequence of numbers 
of forms $\prod_{s=0}^{p^n-1}(1-\xi^s z^{1/p^n})^{s^{k-1}/p^n}$ 
$(\xi\in\mu_{p^n}, n\ge 1)$.
This enabled us to express $\Li_k(z)_\gamma(\sigma)$ ($\sigma\in G_F$)
by the moment integral 
$\frac{1}{(k-1)!}\int_{\Z_p} b^{k-1} d\kk_{1,p}(\sigma)(b)$
over the $p$-adic measure $\kk_{1,p}(\sigma)$ which is 
by definition 
the image of the Kummer-Heisenberg measure
$\kk_1(\sigma)$ (\S \ref{sec.KH}) by the projection
$\hat\Z[[\hat\Z]]\to \Z_p[[\Z_p]]$.

A generalization of this phenomenon has been investigated in
\cite{NW3} for some more general $\Q_p$-paths of the form
$\gamma x^{a/m}$.
We summarize the result as follows:

\begin{Proposition}[\cite{NW3} \S 7]  
\label{moment-int}
Let $\gamma:\ovec{01}\dashto z$ be a pro-$p$ path.
Then,
for any $\alpha\in\Q_p$, we have
$$
\Li_k(z)_{\gamma x^{\alpha}}(\sigma)
=
\frac{1}{(k-1)!}
\int_{\Z_p} (b+\alpha\chi(\sigma))^{k-1} d\kk_{1,p}
(\pathto{\ovec{01}}{\gamma}{z}) (b)
\qquad (\sigma\in G_F).  
$$
\end{Proposition}

\begin{proof}
We just translate \cite[\S 7]{NW3} from $e$-form to $t$-form
in the terminology of Appendix B. 
In $e$-form, it reads (with $\delta:=\gamma$, 
$\alpha:=-\frac{s}{n}$, $\hat\kappa_p:=\boldsymbol{\kappa}_{z,\gamma}$ in \cite[\S 7]{NW3}):
$$
\tilbchi_k^{\ttx^{\alpha}\delta}(\sigma)=
\int_{\Z_p}(b-\alpha\chi(\sigma))^{k-1}d\hat\kappa_p(\sigma)(b).
$$
Let $\gamma x^\alpha$ be the $t$-path 
reciprocally corresponding to the $e$-path $\ttx^\alpha\delta$.
In RHS, we regard the measure $\hat\kappa_p$ as the $p$-adic image of 
$\hat\kappa(\delta)$ of Remark \ref{remark.KH} which can be switched into
the $e$-form $\kk_{1,p}(\gamma)$ to obtain
$$
\int_{\Z_p}(b-\alpha\chi(\sigma))^{k-1}d\hat\kappa_p(\sigma)(b)
=\int_{\Z_p}(-b-\alpha\chi(\sigma))^{k-1}d\kk_{1,p}(\sigma)(b).
$$
At the same time, we may 
convert the LHS to $t$-form by (\ref{A.11}) and 
(\ref{tilbcji-Li}) as
\begin{align*}
\tilbchi_k^{\ttx^{\alpha}\delta}(\sigma)
&=-(k-1)!\, \scLi_k(z)_{\ttx^{\alpha}\delta:\ovec{01}\nyoroto z}(\sigma) \\
&=(-1)^{k-1}(k-1)! \,\Li_k(z)_{\gamma x^\alpha:\ovec{01}\dashto z}(\sigma).
\end{align*}
The formula of the proposition follows from combination of these
identities. 
\end{proof}

%

\section{Consequence of Inversion Formula}

\subsection{Pro-$p$ inversion formula}
We start this section with the main technical result. 

Let $a,m$ be integers with $m>1$, $m\nmid a$, and 
fix the $m$-th root of unity $z:=\xi_m^a \in\mu_m$ and 
set $F=\Q(z)$. 
Pick any path $\gamma:\ovec{01}\dashto z$ 
in $\pi_1^{\proP}(\bold P^1-\{0,1,\infty\},z,\ovec{01})$
and let 
$\bar\gamma:\ovec{01}\dashto z^{-1}$ be the associated path
defined in \S \ref{sec3.1}.

By the assumption $z\in\mu_m$, using
the $p$-adic cyclotomic character $\chi_p:G_F\to\Z_p^\times$, 
we may suppose 
that the Kummer 1-cocycle 
$\rho_{z,\gamma}:G_F\to\Z_p$ 
(written just $\rho_z$ for simplicity) is of a 
1-coboundary form
\begin{equation} \label{eq5.1}
\rho_{z,\gamma}(\sigma)=\rho_z(\sigma)=\alpha(\chi_p(\sigma)-1)
\quad (\sigma\in G_F)
\end{equation}
with a unique constant $\alpha\in \frac{a}{m}+\Z_p$.
Since we do not assume $p\nmid m$, the constant
$\alpha\in \Q_p$ may generally have denominator,
while $\rho_z(\sigma)\in\Z_p$.


\begin{Theorem}
\label{E1}
Notations being as above, we have 
$$
\Li_k(\xi_m^{-a})_{\bar\gamma x^{\alpha}}(\sigma)
+(-1)^k
\Li_k(\xi_m^a)_{\gamma x^{-\alpha}}(\sigma)
=
\frac{1}{k!}B_k(\alpha)(1-\chi_p(\sigma)^k).
$$
for $\sigma\in G_F$ and $k\ge 1$.
\end{Theorem}

This result generalizes \cite[Theorem 10.2]{W3}, where
only the case $p\nmid m$ was considered. 
Here, we shall present a proof 
using the inversion formula for $p$-adic Galois 
polylogarithms \cite{NW2}.
For $\sigma\in G_F$, consider
the $\ell$-adic polylogarithmic characters (for $\ell=p$)
$\tilchi_k(z)_\gamma(\sigma)$, 
$\tilchi_k({\textstyle\frac{1}{z}})_{\bar\gamma}(\sigma)$ 
along those 
pro-$p$ paths
$\gamma$ and $\bar\gamma$.
In \cite[6.3]{NW2}, we showed an inversion formula
for $\gamma$ and $\bar\gamma$
in the following 
form\footnote[1]{See (\ref{tilbchi-tform}). 
The path $\bar\gamma$ from $\ovec{01}$ to 
$z^{-1}\in \mu_m$ in $t$-form here
reciprocally corresponds to the path 
$\la 0,1\ra [1_0^\infty]\la 1, \infty\ra\cdot f_2(\gamma)$ in $e$-form 
with the notation of \cite[\S 6.3]{NW2}.}  
\begin{equation} \label{inversion}
\tilchi_k(z)_\gamma(\sigma)
+(-1)^k\tilchi_k({\textstyle\frac{1}{z}})_{\bar\gamma}(\sigma)
=-\frac{1}{k}\{B_k(-\rho_z(\sigma))-B_k\cdot\chi_p(\sigma)^k\}
\qquad (\sigma\in G_F),
\end{equation}
where $B_k(T)$ is the Bernoulli polynomial defined by
$\sum_{k=0}^\infty B_k(T)\frac{w^k}{k!}=\frac{we^{Tw}}{e^w-1}$
and $B_k=B_k(0)$.
Apply to (\ref{inversion}) the translation formula
\begin{equation}
\Li_k(z)_\gamma(\sigma)=
\frac{(-1)^{k-1}}{(k-1)!}  \,
\tilchi_k(z)_{\gamma}(\sigma)
\qquad (\sigma\in G_F,\, k\ge 1)
\end{equation}
for which we refer the reader to (\ref{A.11}),
(\ref{tilbcji-Li}) and (\ref{tilbchi-tform}),
and obtain
\begin{equation}
\label{Li-inversion}
\Li_k({\textstyle\frac{1}{z}})_{\bar\gamma}(\sigma)
+
(-1)^k \Li_k(z)_{\gamma}(\sigma)
=\frac{1}{k!}(B_k(-\rho_z(\sigma))-B_k \cdot\chi_p(\sigma)^k)
\qquad (\sigma\in G_F).
\end{equation}
Observe that this formula already gives a special case of 
Theorem \ref{E1} where $\alpha=0$ and $\rho_z(\sigma)=0$.
What we shall do from now is to deform this formula into a form 
involved with the $\Q_p$-paths 
$\gamma x^{-\alpha}$ and $\bar\gamma x^{\alpha}$.
In fact, it follows from Proposition \ref{moment-int}, 
we generally have
\begin{equation}
\label{eq4.5}
\Li_k(z)_{\gamma x^\alpha}(\sigma)
=
\sum_{i=0}^{k-1}
{\Li_{k-i}(z)_\gamma(\sigma)}
\frac{(\alpha\chi_p(\sigma))^i }{i!}, \notag
\end{equation}
hence the LHS of Theorem \ref{E1} can be written as:
\begin{align}
\label{eq4.6}
&\Li_k({\textstyle\frac{1}{z}})_{\bar \gamma x^{\alpha}}(\sigma)
+(-1)^{k}
\Li_k(z)_{\gamma x^{-\alpha}}(\sigma) \\
&=
\sum_{i=0}^{k-1}
\frac{(\alpha\chi_p(\sigma))^i }{i!}
\biggl(
\Li_{k-i}({\textstyle\frac{1}{z}})_{\bar \gamma}(\sigma)
+(-1)^{k-i}
\Li_{k-i}(z)_{\gamma}(\sigma) 
\biggr).
\notag
\end{align}


\medskip
To complete the proof of Theorem \ref{E1}, by comparing 
(\ref{Li-inversion}) and (\ref{eq4.6}), 
we are now reduced to 
the following core lemma: 

\begin{Lemma} \label{corevalue} \label{letter4}
Let $k$ be a positive integer, and set
$J_s:=-\frac{1}{s}\{B_s(-\rho_z(\sigma))-B_s \cdot\chi_p(\sigma)^s
\}$
for $s=1,\dots,k$ and $\sigma\in G_F$.
Then, we have
\begin{equation*}
\frac{1}{k}B_k(\alpha) \left(\chi_p(\sigma)^k-1\right)
=\sum_{i=0}^{k-1}
\binom{k-1}{i}
\alpha^i\chi_p(\sigma)^i J_{k-i}(\sigma).
\end{equation*}
\end{Lemma}

\begin{proof}
For simplicity, we omit $\sigma$ in this proof. 
To simplify the RHS of the lemma,
we use the Bernoulli addition formula
\begin{equation} \label{bernoulli_addition} 
B_k(y+x)=\sum_{s=0}^k\binom{k}{s} B_s(y)x^{k-s}. 
\end{equation}
Applying (\ref{bernoulli_addition}) 
with $x=\alpha\chi_p$, $y=-\rho_z$ so that $x+y=\alpha$ by (\ref{eq5.1})
(resp. with $x=\alpha$, $y=0$ so that $x+y=\alpha$), 
we obtain:
$$
\begin{cases}
\displaystyle
B_k(\alpha)&= 
\displaystyle
\sum_{i=0}^{k-1}\binom{k}{i}B_{k-i}(-\rho_z)(\alpha\chi_p)^{i}
+(\alpha\chi_p)^{k} ,\\
\displaystyle
B_k(\alpha) & =
\displaystyle
\sum_{i=0}^{k-1}\binom{k}{i}\alpha^iB_{k-i}+\alpha^k .
\end{cases}
$$
Each of the above identities respectively simplifies 
the former and latter term of the following computation
of the RHS of the lemma. In fact,
noting $\frac{1}{k-i}\binom{k-1}{i}=\frac{1}{k}\binom{k}{i}$, 
one computes:
\begin{align*}
\text{RHS}&=
\left[
-\sum_{i=0}^{k-1}\frac{(\alpha \chi_p)^i}{k}\binom{k}{i}B_{k-i}(-\rho_z)
\right]
+
\left[
\frac{\chi_p^k}{k}\sum_{i=0}^{k-1}\binom{k}{i}\alpha^iB_{k-i} 
\right]  \\
&=
\left[
-\frac{1}{k}B_k(\alpha)+\frac{1}{k}\alpha^k\chi_p^k
\right]
+
\left[
\frac{\chi_p^k}{k}(B_k(\alpha)-\alpha^k)
\right]
\end{align*}
which coincides with the LHS of the aimed identity.
\end{proof}

Thus, our proof of Theorem \ref{E1} is completed. \qed

\begin{Remark}
We mention that the proof in \cite{W3} also carries over in
the case $p\mid m$;  
One needs only consider rational paths in
$\pi^\proP_1(V_1;\xi_m^a,\ovec{01})\hat\otimes \Q$
and in 
$\pi_1^\proP(V_1,\ovec{01})\hat\otimes \Q$.
The embedding of the latter $\pi_1\hat\otimes\Q$
into $\Q_p\lala X,Y\rara$
extends to that of the former pro-$p$
path space into $\Q_p\lala X,Y\rara$.
\end{Remark}

\begin{Remark}
\label{rem5.4}
If $p\nmid m$, then for any $\alpha\in\Z_p$ there is 
$\gamma\in\pi^\proP_1(V_1;\xi_m^a,\ovec{01})$
such that $\rho_{z,\gamma}=\alpha(\chi_p -1)$.
Hence we have
$$
\frac{(k-1)!}{\chi_p(\sigma)^k-1}
\left(\Li_k(\xi_m^{-a})_{\bar\gamma x^{\alpha}}(\sigma)
+(-1)^k
\Li_k(\xi_m^a)_{\gamma x^{-\alpha}}(\sigma)
\right)
=-\frac{B_k(\alpha)}{k}
$$
as long as $\chi_p(\sigma)^k\ne 1$.
A key observation here is the following: 
Taking $\alpha=0$ we get values of the Riemann zeta function
at negative integers (cf.\,\cite{W2}), 
while taking $\alpha=\frac{a}{m}\in\Q^\times$ we get values 
of Hurwitz zeta function 
$\zeta(s,\frac{a}{m})$ at negative integers.
If we choose $\gamma$ from topological paths
$\Gamma_{a/m}\in \pi(V_1(\C);\xi_m^a,\ovec{01})$ 
(\S \ref{sec3.1}), then we get $-\frac{1}{k}B_k(\frac{a}{m})$
for every choice of rational prime $p$. 
\end{Remark}
\subsection{Moment integrals of $p$-adic Hurwitz measure}

First we shall rewrite the formula in
Theorem \ref{E1} in terms of measures 
$\kk_{1,p}(\pathto{\ovec{01}}{\gamma}{\xi_m^a})$
and 
$\kk_{1,p}(\pathto{\ovec{01}}{\bar\gamma}{\xi_m^{-a}})$
after multiplied by $m^{k-1}$.
Set $\alpha:=\frac{a}{m}\in\Q$.
By Proposition \ref{moment-int}
we find that, for $\sigma\in G_F$,
\begin{align}
\label{eq5.8}
m^{k-1}\Li_k(\xi_m^{-a})_{\bar\gamma x^{\alpha}}(\sigma)
&=\frac{m^{k-1}}{(k-1)!}
\int_{\Z_p}(v+\alpha \chi_p(\sigma))^{k-1} 
d\left(\kk_{1,p} (\pathto{\ovec{01}}{\bar\gamma}{\xi_m^{-a}})(\sigma)\right)(v) \\
&= \frac{1}{(k-1)!}  \int_{a\chi_p(\sigma)+m\Z_p}
b^{k-1}
d\left( [m,a\chi_p(\sigma)]_\ast \kk_{1,p}(\xi_m^{-a})_{\bar\gamma}(\sigma)\right)(b) 
\notag \\
&= \frac{1}{(k-1)!}  \int_{\Z_p}
b^{k-1}
d\left( [m,a\chi_p(\sigma)]_\ast \kk_{1,p}(\xi_m^{-a})_{\bar\gamma} (\sigma) \right)(b),
\notag
\end{align}
where the last equality follows as the measure $[m,a\chi_p(\sigma)]_\ast
\kk_{1,p}(\xi_m^{-a})_{\bar\gamma}$ is supported on 
$a \chi_p(\sigma)+m\Z_p\subset \Z_p$.
In the same way, we get that
\begin{align}
\label{eq5.9}
m^{k-1} &
\Li_k(\xi_m^a)_{\gamma x^{-\alpha}}(\sigma)
= \frac{m^{k-1}}{(k-1)!}  \int_{\Z_p}
(v-\alpha\chi_p(\sigma))^{k-1}d
\kk_{1,p}(\pathto{\ovec{01}}{\gamma}{\xi_m^a})(\sigma)(v)
\\ &=
-(-1)^{k-1} 
\frac{m^{k-1}}{(k-1)!}  \int_{\Z_p}
(v+\alpha\chi_p(\sigma))^{k-1}d \left(\iota\cdot
\kk_{1,p}(\ovec{01}\dashto \xi_m^{a})_\gamma (\sigma)
\right) \! (v)
\notag
\\
&= \frac{(-1)^k}{(k-1)!} \int_{\Z_p} b^{k-1}
d\Bigl(
[m,a\chi_p(\sigma)]_\ast \left(
\iota\cdot \kk_{1,p}(\xi_m^{a})_\gamma (\sigma)
\right)\Bigr) (b).
\notag
\end{align}

Now, we enter the situation of Theorem \ref{mainthm1}
and \S 3, that is, $a,m\in\Z$ ($m>1$) are 
integers with $m\nmid a$, and
set $\gamma:=\Gamma_{a/m}$, $\alpha:=a/m$.

\begin{Corollary} \label{MomentIntOverZp}
For the adelic Hurwitz measure 
$\hat\zeta_{a,m}=[m,a\chi_p(\sigma)]_\ast\boldzeta_{a/m}
\in\hat\Z[[\hat\Z]]$, 
the $p$-adic image $\hat\zeta_{p,a.m}(\sigma)
\in\Z_p[[\Z_p]]$ satisfies 
$$
\int_{\Z_p}b^{k-1}d
\hat\zeta_{p,a,m}(\sigma)(b)
=\frac{m^{k-1}}{k} B_k\!\left(\frac{a}{m}\right) 
(1-\chi_p(\sigma)^k)
\quad (\sigma\in G_{\Q(\mu_m)},\ k\ge 2). 
$$
\end{Corollary}

\begin{proof}
Combining the above calculations (\ref{eq5.8}) and (\ref{eq5.9}),
we obtain from Theorem \ref{E1}: 
\begin{align*}
&\frac{m^{k-1}}{k}B_k(\alpha)(1-\chi_p(\sigma)^k)
=(k-1)!\,m^{k-1}\left(
\Li_k(\xi^{-a})_{\bar\gamma x^\alpha}(\sigma)
+(-1)^k \Li_k(\xi^{a})_{\bar\gamma x^{-\alpha}}(\sigma)
\right)
\\
&=\int_{\Z_p} b^{k-1} 
d[m,a\chi_p(\sigma)]_\ast
\left(
\kk_{1,p}(\pathto{\ovec{01}}{\bar\gamma}{\xi_m^{-a}})(\sigma)
+
\iota\cdot
\kk_{1,p}(\pathto{\ovec{01}}{\gamma}{\xi_m^a})(\sigma)
\right)(b)
\\
&=\int_{\Z_p} b^{k-1} 
d[m,a\chi_p(\sigma)]_\ast
\boldzeta_{p,a/m}(\sigma)(b),
\end{align*}
where $\boldzeta_{p,a/m}(\sigma)$ is the image of 
$\boldzeta_{a/m}(\sigma)$ by the projection
$\hat\Z[[\hat\Z]]\to \Z_p[[\Z_p]]$.
This concludes the proof of the corollary.
\end{proof}

%
%
 
\subsection{Proof of Theorem \ref{mainthm1}}

Note first that the support of the 
measure $\hat\zeta_{p,a,m}(\sigma)$ is $a\chi_p(\sigma)+m\Z_p$. 
When $p\mid m$ and $p\nmid a$, it is included in $\Z_p^\times$ so that 
the above Corollary proves the case.

\begin{Remark} \label{smallnote2}
It is worth noting that we do not need to assume $0<a<m$ for the construction
of the measure $\hat\zeta_{a,m}$ and  the integration property in 
the above case of $p\mid m$. This leads to Remark \ref{smallnote1} of Introduction. 
\end{Remark}

The case $p\nmid m$ was treated \cite{W3} in the setting 
where pro-$p$ path $\gamma$ is taken suitably for a fixed $p$.
In our present case, we are taking $\gamma$ to be the 
topological path $\Gamma_{a/m}:\ovec{01}\dashto \xi_m^a$
(cf. Remark \ref{rem5.4}).
We also need the assumption $0<a<m$ for the following

\begin{Lemma}
\label{lem5.6}
Given $m,a\in\Z$, $m>1$ as in Theorem  \ref{mainthm1}, 
suppose that a prime $p$ does not divide $m$.
Let $a_1,\delta\in\Z$ be integers such that $a=pa_1+\delta m$
with $1\le a, a_1< m$.
Then,

(i) $(\Gamma_{a/m})_{\la -\delta\ra,p\ast}=\Gamma_{a_1/m}$;

(ii) $(\bar\Gamma_{a/m})_{\la\delta\ra,p\ast}=\bar\Gamma_{a_1/m}$;

(iii)
$\boldzeta_{a/m}(\sigma)=[p,-\delta \chi(\sigma)]_\ast
\boldzeta_{a_1,m}(\sigma)$ for $\sigma\in G_{\Q(\mu_m)}$.
\end{Lemma}

\begin{proof} (i) results from a good compatibility of our topological paths
$\Gamma_{a/m}$ introduced in \S \ref{sec3.1}
with the lifting along $V_r\twoheadrightarrow V_1$.
Indeed, 
$$
(\Gamma_{a/m}\cdot x^{-\delta})_p
=(\Gamma_{\frac{a}{m}-\delta})_p
=\Gamma_{\frac{a}{pm}-\frac{\delta}{p}}=\Gamma_{a_1/m}
$$ 
which derives (i).
For (ii), suppose $1\le a, a_1< m$. 
Then, noting that $\bar\Gamma_{a/m}$, $\bar\Gamma_{a_1/m}$ are 
homotopic to the complex conjugates of 
$\Gamma_{a/m}$, $\Gamma_{a_1/m}$ respectively (cf.\,Remark \ref{rem3.2}),
we have
\begin{equation*}
(\bar\Gamma_{a/m}\cdot x^\delta)_p
=(\Gamma_{-\frac{a}{m}+\delta})_p
=\Gamma_{\frac{-a+m\delta}{mp}}
=\Gamma_{-a_1/m}=\bar\Gamma_{a_1/m}.
\end{equation*}
This derives (ii).
Finally, using (\ref{magnifiedKappa}), we see from (i) and (ii):
\begin{align*}
\kk_1(\Gamma_{a/m})(\sigma) 
&=
[p,\delta \chi(\sigma)]_\ast \kk_1(\Gamma_{a_1/m}), \\
\kk_1(\bar\Gamma_{a/m})(\sigma) 
&=
[p,-\delta \chi(\sigma)]_\ast \kk_1(\bar\Gamma_{a_1/m}),
\end{align*}
hence from (\ref{iota-translation}) we find
\begin{align*}
[p,-\delta]_\ast \boldzeta_{a_1,m}(\sigma)
&=[p,-\delta]_\ast \left(
\kk_1(\bar\Gamma_{a_1/m})(\sigma)+\iota\cdot \kk_1(\Gamma_{a_1/m})(\sigma)
\right) \\
&=[p,-\delta]_\ast \kk_1(\bar\Gamma_{a_1/m})(\sigma)
+\iota \cdot
[p,\delta]_\ast \kk_1(\Gamma_{a_1/m})(\sigma) \\
&=
\kk_1(\bar \Gamma_{a/m})(\sigma)
+\iota\cdot\kk_1(\Gamma_{a/m})(\sigma) \\
&=\boldzeta_{a/m}(\sigma).
\end{align*}
This settles the proof of (iii).
\end{proof}

Now, we compute the target integral of Theorem \ref{mainthm1}
in the case $p\nmid m$: 
\begin{align*}
&\int_{\Z_p^\times} b^{k-1} 
d[m,a\chi_p(\sigma)]_\ast
\boldzeta_{p,a/m}(\sigma)(b) \\
&=
\int_{\Z_p} b^{k-1} 
d[m,a\chi_p(\sigma)]_\ast
\boldzeta_{p,a/m}(\sigma)(b)
-\int_{p\Z_p} b^{k-1} 
d[m,a\chi_p(\sigma)]_\ast
\boldzeta_{p,a/m}(\sigma)(b), 
\end{align*}
where the first term is calculated as:
\begin{equation}
\label{firstterm}
\int_{\Z_p} b^{k-1} 
d[m,a\chi_p(\sigma)]_\ast
\boldzeta_{p,a/m}(\sigma)(b)
=\frac{m^{k-1}}{k}B_k\left(\frac{a}{m}\right)(1-\chi_p(\sigma)^k)
\end{equation}
by Corollary \ref{MomentIntOverZp}.
For the second term, we observe:
\begin{equation*}
\int_{p\Z_p} b^{k-1} 
d[m,a\chi_p(\sigma)]_\ast
\boldzeta_{p,a/m}(\sigma)(b)
=
\int_S
(mv+a\chi_p(\sigma))^{k-1}
d
\boldzeta_{p,a/m}(\sigma)(v)
\end{equation*}
with $S:=\{v\in\Z_p\mid mv+a\chi_p(\sigma)\in p\Z_p\}$.
Since $p\nmid m$, we can choose integers $a_1,\delta\in\Z$ such that $a=a_1p+m\delta$.
We set $a_1=\la ap^{-1}\ra$ to be the least positive one as introduced in Theorem \ref{mainthm1}.
In this set up, the condition 
$mv+a\chi_p(\sigma)=m(v+\delta\chi_p(\sigma))+p a_1\chi_p(\sigma)\in p\Z_p$ 
is equivalent to the condition $v+\delta\chi_p(\sigma)\in p\Z_p$; hence 
the space $S$ is a coset form: $S=-\delta\chi_p(\sigma)+p\Z_p$.
If $v\in S$ is written as $v=-\delta\chi_p(\sigma)+p\beta$ ($\beta\in\Z_p$), then
$mv+a\chi_p(\sigma)=p(m\beta+a_1)$. 
Noting that Lemma \ref{lem5.6} (iii) implies
$\boldzeta_{p,a/m}(\sigma)=[p,-\delta \chi(\sigma)]_\ast
\boldzeta_{p,a_1,m}(\sigma)$
for the $p$-adic images of measures by $\hat\Z[[\hat\Z]]\to \Z_p[[\Z_p]]$,
we obtain
\begin{align*}
\int_S
(mv+a\chi_p(\sigma))^{k-1}
d
\boldzeta_{p,a/m}(\sigma)(v)
&=
p^{k-1}
\int_{\Z_p} (m\beta+a_1\chi_p(\sigma))^{k-1}
d\boldzeta_{p,a_1,m}(\sigma)(\beta) \\
&=
p^{k-1} \int_{\Z_p} b^{k-1} d[m,a_1\chi_p(\sigma)]_\ast \boldzeta_{p,a_1,m}(\sigma)(b) \\
&=p^{k-1} \int_{\Z_p} b^{k-1} d \hat\zeta_{p,a_1,m}(\sigma)(b) \\
&=\frac{(pm)^{k-1}}{k}B_k\left(\frac{a_1}{m}\right)(1-\chi_p(\sigma)^k),
\end{align*}
where the last identity follows from Corollary \ref{MomentIntOverZp}.
This, combined with (\ref{firstterm}), settles the remained case of 
Theorem \ref{mainthm1}.
\qed

%

\appendix

\section{Cohen's $p$-adic Hurwitz zeta function}
\label{CohenApp}

In this appendix, we shall relate 
the $p$-adic Hurwtiz zeta function $\zeta_p(s,x)$
introduced in H. Cohen's book \cite{Co} 
to
the $p$-adic Hurwtiz zeta function
of Shiratani type (\cite{Sh}) discussed in our main text. 
Let $p$ be a prime, and let $q=p$ or $q=4$ 
according as $p>2,p=2$ respectively. 
Set $C\Z_p:=\Q_p\setminus \frac{p}{q}\Z_p$.
Cohen's $\zeta_p(s,x)$ is defined first in \cite[\S 11.2.2]{Co}
for $x\in C\Z_p$ and is then defined
also for $x\in \Z_p$ in \cite[\S 11.2.4]{Co}.
Our main goal is to give connection with $\zeta_p(s,\frac{a}{m})$ in
the formulas (\ref{cohen-shiratani1}) and (\ref{cohen-shiratani2}).

\medskip
\noindent
{\bf The case {$\zeta_p(s,x)$ for $x\in C\Z_p$}:}
\smallskip

In \cite[Theorem 5.9]{Wa}, a $p$-adic meromorphic function 
$H_p(s,a,m)$ in $s$ is introduced for a pair of integers $a,m$ 
such that $0<a<m$, $q\mid m$ and $p\nmid a$. 
It satisfies
\begin{equation}
H_p(1-k, a,m)=-\omega(a)^{-k}\frac{m^{k-1}}{k}B_k(\frac{a}{m})
\quad (k\in\N),
\end{equation}
where $\omega:\Z_p^\times\to\mu_e$ ($e:=|(\Z/q\Z)^\times|$)
is the $p$-adic Teichm\"uller character.
Cohen extends $\omega$ to $\omega_v:\Q_p^\times\to p^\Z\cdot \mu_e$
by $\omega_v(up^n)=p^n\omega(u)$ $(u\in\Z_p^\times$, $n\in\Z)$.
Then, the interpolation property of $\zeta_p(s,x)$ for $x\in C\Z_p$ 
given in \cite[Theorem 11.2.9]{Co} reads
$$
\zeta_p(1-k,x)=-\omega_v(x)^{-k}\frac{B_k(x)}{k}  \quad (k\in \N).
$$
This specializes for $x=a/m\in\Q\cap C\Z_p$
($p\nmid a$, $q\mid m$ without assuming $0<a<m$)
to 
\begin{equation}
\zeta_p(1-k,\frac{a}{m})
=-\omega_v(m)^k\cdot \omega(a)^{-k}\frac{1}{k}B_k(\frac{a}{m}) .
\end{equation}
Restricting $k$ to those positive integers in a same class in $\Z/e\Z$, we obtain a relation
between special values of $\zeta_p(s,a/m)$ and $\{L_p^{[\beta]}(s;a,m)\}_{\beta\in\Z/e\Z}$ of
Remark \ref{p-adicLbeta}:
For $k\equiv\beta(\text{mod } e), \, k\ge 1$, writing $m=p^{v_p(m)}m_1$, we have
\begin{align}
\zeta_p(1-k,\frac{a}{m}) &=-
\left(
\frac{\omega_v(m)}{\omega(a)} \right)^k
m^{1-k}
L_p^{[\beta]}(1-k;a,m)
\\
&=-\left(
\frac{\omega_v(m)}{\omega(a)^{\beta}} \right)
\left(
\frac{\omega_v(m)}{m} \right)^{k-1}
L_p^{[\beta]}(1-k;a,m)  \notag
\\
&=
-\left(\frac{\omega_v(m)}{\omega(a)^{\beta}} \right)
\left(
\frac{\omega(m_1)^{\beta-1}}{m_1^{k-1}} 
\right)
L_p^{[\beta]}(1-k;a,m). \notag
\end{align}
Hence, under the assumption
$q\mid m$ and $p\nmid a$,
for any $\beta\in\Z/e\Z$, it follows that
\begin{equation}
\label{cohen-shiratani1}
\zeta_p(s,\frac{a}{m})=-\left(
\frac{\omega_v(m)}{\omega(a)^{\beta}} \right)
\left(
\frac{\omega(m_1)^{\beta-1}}{m_1^{-s}} 
\right)
L_p^{[\beta]}(s;a,m)
\end{equation}
for $s$ in the space 
$\beta+\frac{q}{p}\Z_p$ which is one of the forms $\Z_p$ $(p>2)$, 
$2\Z_2$ or $1+2\Z_2$.
Due to Remark \ref{smallnote1}, 
this formula holds true for all $a\in\Z$ with $m\nmid a$
and $p\nmid a$.

\medskip
\noindent
{\bf The case {$\zeta_p(s,x)$ for $x\in\Z_p$}:}
\smallskip

In this case, let us first observe the following identity:
\begin{equation} \label{cohen}
\zeta_p(1-k,x)=-\frac{1}{k}B_k(\tilde{\omega}^{-k},x)
\qquad (k\in\Z_{\ge 1})
\end{equation}
where $\tilde{\omega}$ is the Teichm\"uller character on $\Z_p$
(extended by $0$ on $p\Z_p$) 
and $B_k(\chi,\ast)$ is the $\chi$-Bernoulli polynomial
defined in \cite[\S 9.4.1]{Co} .


\begin{proof}
Write $x=p^u\alpha\in\Z_p$ with $p\nmid \alpha$ and set $N=p^v=p^{u+1}$.
We use \cite[Corollary 11.2.15]{Co} and the notations there.
As $\omega_v(N)=p^v$, $\la N\ra=1$, we have for $s=1-k$:
\begin{equation*}
p^v \cdot \zeta_p(1-k, x)
=\sum_{\substack{0\le j<p^v \\ p\nmid j}}
\zeta_p(1-k, \frac{x}{p^v}+\frac{j}{p^v}).
\end{equation*}
In RHS here, it follows from \cite[Theorem 11.2.9]{Co} that
$$
\zeta_p(1-k, \frac{x}{p^v}+\frac{j}{p^v})
=-(p^{-v}\omega(j))^{-k}\frac{1}{k}B_k(\frac{x}{p^v}+\frac{j}{p^v}).
$$
Hence
\begin{align*}
p^v \cdot \zeta_p(1-k, x)
&=
-\frac{p^{kv}}{k}\cdot\sum_{j=0}^{p^v}\tilde{\omega}^{-k} (j) \cdot
B_k(\frac{x}{p^v}+\frac{j}{p^v}) \\
&=
-\frac{p^{kv}}{k}\cdot p^{v(1-k)}
B_k(\tilde{\omega}^{-k},x)
\quad  (\text{by \cite[Lemma 9.4.7]{Co}
}).
\end{align*}
This proves (\ref{cohen}).
\end{proof}

Let $e=|(\Z/q\Z)^\times|$. Then,
\begin{equation} \label{cohen2}
\zeta_p(1-k,x)=-\frac{1}{k}
\left(
B_k(x)-p^{k-1}B_k(\frac{x}{p})
\right) 
\end{equation}
for  $k\in\Z_{\ge 1}$ and $k\equiv 0$ mod $e$.

\begin{proof}
When $k\equiv 0$ mod $e$, we have $\tilde{\omega}^k(j)=0,1$
according as $p\mid j$, $p\nmid j$.
Putting this into the basic identity of $\chi$-Bernoulli polynomial
(\cite[Proposition 9.4.5]{Co}), we find
$$
B_k(\tilde{\omega}^{-k}, x)
=p^{k-1}\sum_{j=0}^{p-1} \tilde{\omega}(j)^{-k} B_k(\frac{x+j}{p})
=p^{k-1}\sum_{j=1}^{p-1} B_k(\frac{x+j}{p}).
$$
On the other hand, 
from the usual distribution formula of the Bernoulli polynomial:
it follows that
$$
B_k(x)=p^{k-1}\sum_{j=0}^{p-1}B_k(\frac{x}{p}+\frac{j}{p}).
$$
By comparing these two identities, we obtain (\ref{cohen2}).
\end{proof}


Suppose $x=\frac{a}{m}\in\Q\cap\Z_p$ with $p \nmid m$. 
Observe that the above interpolation property of Cohen's $\zeta_p(s,x)$ 
(\ref{cohen2}) reads then
\begin{equation} \label{cohen3}
\zeta_p(1-k,\frac{a}{m})=-\frac{1}{k}
\left(
B_k(\frac{a}{m})-p^{k-1}B_k(\frac{a}{mp})
\right) 
\qquad (k\ge 1,\ k\equiv 0(\text{mod } e)).
\end{equation}
It is not straightforward to find a connection from this to 
the interpolation property of 
Shiratani's $\zeta_p^{Sh}(s,a,m)=-L_p^{[0]}(s;a,m)$ (cf. Remark \ref{p-adicLbeta}):
\begin{equation}\label{shiratani3}
\zeta_p^{Sh}(1-k;a,m)
=-\frac{m^{k-1}}{k}\left(B_k(\frac{a}{m})-p^{k-1}B_k(\frac{\la ap^{-1}\ra }{m})\right)
\end{equation}
for all $k\in\Z_{>0}$, $k\equiv 0$ mod $e$, where 
$a$ and $\la ap^{-1}\ra$ are supposed to be positive integer $<m$ such that
$\la ap^{-1}\ra p\equiv a\mod m$.
To connect (\ref{cohen3}) and (\ref{shiratani3}), let $r$ be the {\it unique} integer
with $a+mr=\la ap^{-1}\ra p$
so that $\frac{a+mr}{mp}=\frac{\la ap^{-1}\ra}{m}$.
Note that $r>0$, due to the assumption $0<a, \la ap^{-1}\ra <m$.
Then, replacing $a$ by $a+mr$ in (\ref{cohen3}), we find
\begin{align*}
\zeta_p(1-k,\frac{a+mr}{m})-m^{1-k}\zeta_p^{Sh}(1-k;a,m)
&=-\frac{1}{k}B_k(r+\frac{a}{m})+\frac{1}{k}B_k(\frac{a}{m}) \\
&=
-\sum_{v=0}^{r-1} (\frac{a+mv}{m})^{k-1}
\end{align*}
for all 
$k>0,\,k\equiv 0\ \mathrm{mod}\,{e}$.
We claim below that $p\nmid (a+mv)$ for all $v\in [0,r-1]$
so that the existence of
$p$-adic analytic functions $(\frac{a}{m}+v)^s$ 
provides a 
connection between Cohen's $\zeta_p(s,x)$ with $x\in\Z_p$ and Shiratani's 
$\zeta_p^{Sh}(s;a,m)=-L_p^{[0]}(s;a,m)$, namely, it holds that
\begin{equation}
\label{cohen-shiratani2}
\zeta_p(s,\frac{a+mr}{m})=
m^{s}\cdot \zeta_p^{Sh}(s;a,m)
- \sum_{v=0}^{r-1} (\frac{a+mv}{m})^{-s}
\end{equation}
for $s\in 
\frac{q}{p}\Z_p$,
under the assumptions $p\nmid m$, $0<a<m$, $p\mid(a+mr)$ and $0<a+mr<pm$.
The assertion (\ref{cohen-shiratani2}) is thus 
reduced to the following elementary

\begin{Claim*} Notations being as above, let $r_0$ be the least 
nonnegative integer such that $a+mr_0\equiv 0\mod p$. Then, 
$a+mr_0=\la ap^{-1}\ra p$.
\end{Claim*}

\begin{proof}
If $r_0\ge p$, then $a+m(r_0-p)\equiv a+mr_0\equiv 0\mod p$ which
contradicts the minimality of $r_0$. Therefore $r_0<p$.
If $r_0=p-1$, then writing $a+m(p-1)=xp$, we have
$p(m-x)=m-a>0$. Hence $m>x>0$, i.e., $x=\la ap^{-1}\ra$.
Assume that $r_0<p-1$. 
If $mp \le a+mr_0$, then since $a<m$, it follows that
$0\le a+m(r_0-p)<m(r_0-p+1)$ hence that $p-1\le r_0$
contradicting the assumption. 
Thus $mp>a+mr_0$. Writing $a+mr_0=xp$, we obtain $m>x$,
that is, $x=\la ap^{-1}\ra$.
\end{proof}

\begin{Example*}
Let $p=11$, $a=3$ and $m=106$. Noting 
$106\equiv 7\mod 11$ and $3+7\cdot 9=6\cdot 11$, one finds 
$3+106\cdot 9=957=87\cdot 11$. Hence $\la 3\cdot 11^{-1}\ra=87$.
Now, the core sum in the above construction reads
$\sum_v (\frac{a}{m}+v)^{k-1}=
(\frac{3}{106})^{k-1}+
(\frac{109}{106})^{k-1}+
(\frac{215}{106})^{k-1}+
(\frac{321}{106})^{k-1}+
(\frac{427}{106})^{k-1}+
(\frac{533}{106})^{k-1}+
(\frac{639}{106})^{k-1}+
(\frac{745}{106})^{k-1}+
(\frac{851}{106})^{k-1}.
$
There do exist $11$-adic analytic functions
that interpolate $(\frac{3+106v}{106})^{k-1}$
at $s=1-k$ ($k\equiv 0\mod 10$) for
$v=0,1,.\dots,8$ respectively.
\end{Example*}

\begin{Question*}
It is unclear if $L_p^{[\beta]}(s;a,m)$ for $p\nmid m$, $\beta\not\equiv 0$ ($e$) can
be expressed in terms of Cohen's $\zeta_p(s,x)$.
\end{Question*}

\section{Path conventions}  \label{PathApp}

In this Appendix, we quickly summarize two conventions on
\'etale paths mostly used in our papers. 
Just for simplicity, we call one system of conventions 
the traditional form (`$t$-form') and 
another system the electronic form (`$e$-form').
The present paper and most papers by the second author
obey the $t$-form, whereas most papers by the first author
and our previous common papers [NW1-3] obey the $e$-form.
The purpose of this Appendix is to serve a dictionary to
translate formulas between these two forms. 

Let $\cC$ be a Galois category, for example, that of the
finite \'etale covers of an algebraic variety. We write $a,b,c,...$ for general symbols playing
roles of base points for $\pi_1(\cC)$ and $\omega_a,\omega_b,\omega_c,\dots$
for the corresponding Galois functors $\cC\to Sets$.
The path space between two points $a$ and $b$ is by definition the set 
$\Isom(\omega_a,\omega_b)$ whose element is a compatible family
of isomorphisms of fibre sets $\gamma_U: \omega_a(U)\isom\omega_b(U)$ over 
$U\in \mathrm{Ob}(\cC)$.
In $t$-form, an element $\gamma$ 
of $\Isom(\omega_a,\omega_b)$ is called 
a ($t$-)path from $a$ to $b$ and written as
$\gamma:a \dashto b$. 
In $e$-form, the same $\gamma\in\Isom(\omega_a,\omega_b)$ is called a(n $e$-)path
\underline{from $b$ to $a$} and written as 
$\gamma:b \nyoroto a$.
Remind that, for each $U\in\mathrm{Ob}(\cC)$, 
$\gamma_U(s)$ is defined for elements $s\in\omega_a(U)$. 
[In $e$-form, we may imagine that the waving arrow $\gamma:b\nyoroto a$ 
flows like an electronic current that conveys electron $s\in\omega_a(U)$
back into $\omega_b(U)$.]
We shall use the notation
\begin{equation}
\pi_1(\cC;b,a):=\Isom(\omega_a,\omega_b)
\end{equation}
to designate the set of $t$-paths from $a$ to $b$ 
as well
as the set of $e$-paths from $b$ to $a$.
Accordingly, if $\gamma_1\in \Isom(\omega_a,\omega_b)$ and 
$\gamma_2\in \Isom(\omega_b,\omega_c)$, then
the composite $\gamma_2\gamma_1\in\Isom(\omega_a,\omega_c)$
is defined. We have
\begin{align}
[a\overset{\gamma_2\gamma_1\ }{-\dashto c}]
&=
[b\overset{\gamma_2}{\dashto} c]
\cdot [a\overset{\gamma_1}{\dashto} b] \ 
\left(\text{viz. }
[c\overset{\gamma_2\gamma_1\ }{\dashleftarrow\!- a}]
=[c\overset{\gamma_2}{\dashleftarrow} b]\cdot
[b\overset{\gamma_1}{\dashleftarrow} a]
\right), 
\\
[c\overset{\gamma_2\gamma_1}{\nyoroto\!\nyoroto}a]
&=
[c\overset{\gamma_2}{\nyoroto} b]\cdot [b\overset{\gamma_1}{\nyoroto} a]. 
\end{align}

Next, let $F$ be a subfield of $\C$ and let 
$\cC$ be the Galois category of finite \'etale covers of 
an algebraic variety $V$ over $F$.
If $a$ is an $F$-rational (tangential) points on $V$, then
the sequence of finite sets $\{\omega_a(U)\}_{U\in\mathrm{Ob}(\cC)}$
have compatible actions by $G_F$, which defines
the map $G_F\to \Isom(\omega_a,\omega_a)=\pi_1(V,a)$.
For two such points $a,b$, we define the canonical 
left $G_F$-action on 
$\Isom(\omega_a,\omega_b)$ by $\gamma\mapsto
\sigma\gamma\sigma^{-1}$ $(\sigma\in G_F)$.
Observe that, concerning Galois actions, no difference occurs 
between $t$-form and $e$-form.

Suppose now that
$V=\mathbf{P}^1_\Q-\{0,1,\infty\}$. 
Denote by $x$, $y$ the standard loops based at $\ovec{01}$ running around
the punctures $0$, $1$ respectively with anticlockwise $t$-arrows $\dashto$, 
and let $\ttx$, $\tty$ be those loops with anticlockwise $e$-arrows
$\nyoroto$.
Then, $x=\ttx^{-1}$, $y=\tty^{-1}$. 
Let $z$ be a $F$-rational (tangential) point on $V$.
For a $t$-path $\gamma:\ovec{01}\dashto z$ on $V\otimes F$,
we define a Galois associator in $t$-form by
\begin{equation}
\wf_\gamma(\sigma):=\gamma^{-1}\cdot\sigma(\gamma)
\in\pi_1^{\text{\'et}}(V,\ovec{01})
\qquad (\sigma\in G_F)
\end{equation}
as in \S 2 of the present paper, 
whereas, for an $e$-path $\delta:\ovec{01}\nyoroto z$ on $V\otimes F$,
we define another Galois associator in $e$-form by 
\begin{equation}
\ff^{\delta}_\sigma:=
\delta\cdot\sigma(\delta)^{-1}
\in\pi_1^{\text{\'et}}(V,\ovec{01})
\qquad (\sigma\in G_F)
\end{equation}
as in \cite{NW1}. Therefore, assuming $\delta=\gamma^{-1}$,
we find
\begin{equation}
\label{fandf}
\wf_\gamma(\sigma)=\ff^{\delta}_\sigma
\qquad (\sigma\in G_F).
\end{equation}

Given a prime $\ell$, let $\pi_{\Q_\ell}$ be the pro-unipotent
completion of the maximal pro-$\ell$ quotient of $\pi_1(V_{\overline\Q},\ovec{01})$.
Consider the above $x,y,\ttx,\tty$ as $\Q_\ell$-loops based at $\ovec{01}$
on $V$, 
and regard $\gamma:\ovec{01}\dashto z$ and $\delta:\ovec{01}\nyoroto z$
as $\Q_\ell$-paths on $V$.
If $\delta=\gamma^{-1}$, then $(\gamma x^\alpha)^{-1}=\ttx^\alpha \delta$; hence
it follows from (\ref{fandf}) that
\begin{equation}
\wf_{\gamma x^\alpha}(\sigma)=\ff^{\,\ttx^\alpha \delta}_\sigma
\qquad 
(\delta=\gamma^{-1},\ \sigma\in G_F, \ \alpha\in\Q_\ell).
\end{equation}

Now, let us compare $\ell$-adic Galois polylogarithms in $t$-form and $e$-form.
Define generators $X,Y$ (resp. $\bar X,\bar Y$) of $\mathrm{Lie}(\pi_{\Q_\ell})$
so that $e^X,e^Y$ (resp. $e^{\bar X}, e^{\bar Y}$) 
are the $\ell$-adic images of $x,y$ (resp. $\ttx,\tty$) in $\pi_{\Q_\ell}$.
Then $X=-\bar X$, $Y=-\bar Y$.
Let $I_Y$ (resp. $I_{\bar Y}$) denote the ideal of $\mathrm{Lie}(\pi_{\Q_\ell})$ generated by those
Lie words in $X,Y$ (resp. $\bar X, \bar Y$) containing $Y$ (resp. $\bar Y$) twice or more. 
Obviously we have $I_Y=I_{\bar Y}\subset\mathrm{Lie}(\pi_{\Q_\ell})$.
In $e$-form, we have the Lie expansion 
\begin{equation}
\label{Lieff}
\log (\ff^{\delta}_\sigma)^{-1}
\equiv \rho_{z,\delta}(\sigma) \bar X 
+ \sum_{k=1}^\infty \ell i_k(z)_\delta(\sigma) 
(\mathrm{ad}\, \bar X)^{k-1}(\bar Y)
\mod I_{\bar Y}
\end{equation}
extending \cite[Definition 5.4]{NW2} to any $\Q_\ell$-paths $\delta:\ovec{01}\nyoroto z$.
Note that interpretation of $\rho_{z,\delta}$ as a Kummer 1-cocycle along 
power roots of $z$ is basically available only when $\gamma$ is a pro-$\ell$ path.
On the side of $t$-form, one also has
\begin{equation}
\label{LieWf}
\log (\wf_\gamma(\sigma)) \equiv
\rho_{z,\gamma}(\sigma) X
+ \sum_{k=1}^\infty \ell i_k(z)_\gamma(\sigma) 
[..[Y,\underbrace{X],\dots,X}_{k-1}] \mod I_Y
\end{equation}
extending \cite[Definition 11.0.1]{W1} for any $\Q_\ell$-path $\gamma:\ovec{01}\dashto z$.
Comparing (\ref{Lieff}) and (\ref{LieWf}) under the situation (\ref{fandf}),
we see that the $\rho_z$ and the $\ell$-adic polylogarithms $\ell i_m(z)$ (written also as
$\ell_m(z)$ in older papers)
for $\gamma:\ovec{01}\dashto z$ in $t$-form and 
for $\delta:\ovec{01}\nyoroto z$ in $e$-form 
coincide with each other as functions on $G_F$
as long as $\delta=\gamma^{-1}$, that is,
\begin{equation}
\rho_{z,\gamma}(\sigma)=\rho_{z,\delta}(\sigma),\ 
\ell i_k(z)_\gamma(\sigma) = \ell i_k(z)_\delta(\sigma) 
\qquad(\sigma\in G_F,\, k\ge 1,\,\delta=\gamma^{-1}).
\end{equation}
Next, embed $\mathrm{Lie}(\pi_{\Q_\ell})$ into
the ring of non-commutative power series 
$\Q_\ell\lala X,Y\rara=\Q_\ell \lala \bar X,\bar Y\rara$
and expand $f_\gamma(\sigma)=\ff_\sigma^\delta$ into
series in $X,Y$ or in $\bar X,\bar Y$.
The coefficient at $YX^{k-1}$ appearing in the former expansion
is the 
$\ell$-adic polylogarithm $\Li_k(z)_\gamma(\sigma)$
in \S 4.1 of this paper in $t$-form (with $\ell=p$), while the coefficient
at $\bar Y\bar X^{k-1}$ in the latter expansion,
which we denote by $\scLi_k(z)_\delta(\sigma)$
in $e$-form, was discussed in \cite[\S 6]{NW3}.
By definition, we have
\begin{equation}
\label{A.11}
\Li_k(z)_{\pathto{\ovec{01}}{\gamma}{z}}(\sigma)
=(-1)^k 
\scLi_k(z)_{\ovec{01}\overset{\delta}{\nyoroto}z}(\sigma)
\qquad(\sigma\in G_F,\, k\ge 1,\,\delta=\gamma^{-1}).
\end{equation}
Finally, we recall from \cite[\S 6]{NW3}.
the function
$$
\tilbchi_k^{z,\delta}:G_F\to\Q_\ell
$$
associated to any $\Q_\ell$-path $\delta:\ovec{01}\nyoroto z$
for $k\ge 1$ by the equation:
\begin{equation}
\tilbchi_k^{z,\delta}(\sigma)
=(-1)^{k+1}(k-1)!
\sum_{i=1}^k
\frac{{\rho_{z,\delta}(\sigma)}^{k-i}}{(k+1-i)!}
\ell i_i(z)_\delta(\sigma).
\end{equation}
It is related to the above $\scLi_k(z)_{\ovec{01}\overset{\delta}{\nyoroto}z}(\sigma)$
by
\begin{equation}
\label{tilbcji-Li}
-\frac{\tilbchi_k^{z,\delta}(\sigma)}{(k-1)!}=
\scLi_k(z)_{\ovec{01}\overset{\delta}{\nyoroto}z}(\sigma)
\quad (\sigma\in G_F,\,k\ge 1).
\end{equation}
When $\delta:\ovec{01}\nyoroto z$ is a pro-$\ell$ path, then 
$\tilbchi_k^{z,\delta}$ is the polylogarithmic character
studied in \cite{NW1} and is known to be valued in $\Z_\ell$
with explicit Kummer properties along a sequence of numbers.


For a path $\gamma:\ovec{01}\dashto z$ in $t$-form, we employ the notation
\begin{equation}
\label{tilbchi-tform}
\tilbchi_k(z)_{\gamma}(\sigma):=
\tilbchi_k^{z,\delta}(\sigma) \qquad (\sigma\in G_F)
\end{equation}
where $\delta=\gamma^{-1}:\ovec{01}\nyoroto z$ is the corresponding
path in $e$-form.
It follows then that
\begin{equation}
\frac{\tilbchi_k(z)_{\gamma}(\sigma)}{(k-1)!}
=(-1)^{k-1}\Li_k(z)_{\pathto{\ovec{01}}{\gamma}{z}}(\sigma)
\qquad(\sigma\in G_F)
\end{equation}
for any $\Q_\ell$-path $\gamma:\ovec{01}\dashto z$.


\ifx\undefined\bysame
\newcommand{\bysame}{\leavevmode\hbox to3em{\hrulefill}\,}
\fi

\end{document}